\documentclass[a4paper,11pt,oneside,notitlepage]{article}
\usepackage{times}
\usepackage{lmodern}
\usepackage[T1]{fontenc}
\usepackage[utf8]{inputenc}
\usepackage[english]{babel}
\usepackage{amssymb}
\usepackage{amsmath}
\usepackage{mathtools} % it allows to typeset formula in display style without any commands... \begin{dcases}
\usepackage{amsthm}
\usepackage{amsfonts, mathrsfs}
\usepackage{latexsym}
\usepackage[pdftex]{color, graphicx}
\usepackage{makeidx}
\usepackage{sidecap}
\usepackage{verbatim}
\usepackage{geometry}
\usepackage{subfig}
\usepackage{tikz,float}
\usepackage[mathlines, running, switch, columnwise]{lineno} % number of line
\usepackage{comment} % label interni
\usepackage[colorinlistoftodos]{todonotes}
\usepackage{esint} % per il simbolo della media integrale
\usepackage{pdfsync}
\usepackage{todonotes}
\usepackage{hyperref}

\usetikzlibrary{shapes,arrows,shadows}
\definecolor{vg}{rgb}{0.0, 0.40, 0.15}

\newenvironment{todorev}{\color{cyan}}{\color{black}}
\newcommand{\btodo}{\begin{todorev}}
\newcommand{\etodo}{\end{todorev}}

\newenvironment{samuelerev}{\color{purple}}{\color{black}}
\newcommand{\bsamr}{\begin{samuelerev}}
\newcommand{\esamr}{\end{samuelerev}}

\newenvironment{giacomorev}{\color{red}}{\color{black}}
\newcommand{\bg}{\begin{giacomorev}}
\newcommand{\eg}{\end{giacomorev}}

\newtheorem{theorem}{Theorem}[section]
\newtheorem{proposition}[theorem]{Proposition}
\newtheorem{lemma}[theorem]{Lemma}
\newtheorem{corollary}[theorem]{Corollary}
\newtheorem{definition}[theorem]{Definition}
\newtheorem{remark}[theorem]{Remark}
\newtheorem{example}[theorem]{Example}

\numberwithin{equation}{section}

\def\R{\mathbb{R}}

\def\N{\mathbb{N}}

\def\I{\mathbb{I}_n}

\def\M{\mathcal{M}}

\def\loc{{\rm loc}}

\def\sym{{\rm sym}}
\def\supp{{\rm supp}}

\def\ker{{\rm ker}}
\def\tr{{\rm tr}}

\renewcommand{\div}{\mathrm{div}}
\newcommand{\dom}{\mathrm{dom}}

\newcommand{\T}{\textnormal{T}}

\newcommand{\e}{\varepsilon}
\newcommand{\vphi}{\varphi}
\newcommand{\vrho}{\varrho}
\newcommand{\E}{\mathcal{E}}

\newcommand{\st}{ \ : \ }

\title{Functions of bounded Musielak-Orlicz-type deformation and anisotropic Total Generalized Variation \\ for image-denoising problems}

\author{Giacomo Bertazzoni, Elisa Davoli, Samuele Riccò and Elvira Zappale}

\usepackage[twoside]{fancyhdr}
\AtEndDocument{	
    \newpage
    
    \hrule \medskip
    
    \small{
        \noindent
        {\sc G. Bertazzoni},
         Dipartimento di Scienze Fisiche, Informatiche e Matematiche, Università degli Studi di Modena e Reggio Emilia, via Campi 213/B, 41125 Modena, Italy
         
        \noindent
        E-mail:
        \href{mailto:giacomo.bertazzoni@unimore.it}{\tt giacomo.bertazzoni@unimore.it}.
    }
   
    \medskip \medskip \medskip \hrule \medskip
    
    \small{
		\noindent
		{\sc E. Davoli},
		Institute of Analysis and Scientific Computing,
		TU Wien, Wiedner Hauptstra\ss e 8-10, 1040 Vienna, Austria.
		
		\noindent
		E-mail:
		\href{mailto:elisa.davoli@tuwien.ac.at}{\tt elisa.davoli@tuwien.ac.at}.
    }
   
    \medskip \medskip \medskip \hrule \medskip
   
    \small{
		\noindent
		{\sc S. Ricc\'o},
		Institute of Analysis and Scientific Computing,
		TU Wien, Wiedner Hauptstra\ss e 8-10, 1040 Vienna, Austria

		\noindent
		E-mail:
		\href{mailto:samuele.ricco@tuwien.ac.at}{\tt samuele.ricco@tuwien.ac.at}.
    }
    
    \medskip \medskip \medskip \hrule \medskip
    
    \small{
		\noindent
		{\sc E. Zappale},
		Dipartimento di Scienze di Base e Applicate per l’Ingegneria, Sapienza - Universit\`{a} di Roma,
via Antonio Scarpa, 16, 00161 Roma, Italy

		\noindent
		E-mail:
		\href{mailto:elvira.zappale@uniroma1.it}{\tt elvira.zappale@uniroma1.it}.
    }
    
    \medskip \medskip \medskip \hrule  
}

\begin{document}

\maketitle

\begin{abstract}
    \noindent
    In the first part of this paper we introduce the space of bounded deformation fields with generalized Orlicz growth. We establish their main properties, provide a modular representation, and characterize a decomposition of the modular into an absolutely continuous part and a singular part weighted via a recession function. A further analysis in the variable exponent case is also provided. The second part of the paper contains a notion of Musielak-Orlicz anisotropic Total Generalized Variation. We establish a duality representation, and show well-posedness of the corresponding image reconstruction problem.
    
    \medskip
    
    \noindent
    {\it 2020 Mathematics Subject Classification:} 26A45, 26B30, 46E35, 46E99, 49J27, 49J45
    
    \smallskip
    
    \noindent
    {\it Keywords and phrases:}
    Functions of bounded deformation, generalized Orlicz spaces, modular, Total Generalized Variation, imaging.
    
\end{abstract}

%%%%%%%%%%%%%%%%%%%%%%%%%%%%%%%%%%%%%%%%%%%%%%%%%%%%%%%%%%%%%%%%

\section{Introduction}
\label{intro}

Ever since the early 90s, the variational approach has been a cornerstone of the mathematical image reconstruction. Generally speaking, in its simpler formulation, this methodology can be summarized as follows: for a given bounded domain $\Omega \subset \R^2$, grayscale images are identified with functions $u : \Omega \to \R$. Starting from a datum $f$, possibly corrupted by noise, a clean image is obtained through the minimization of an energy functional of the form
\begin{equation*}
    \mathcal{I}(u) := \alpha \|Ku - f\|_H + \beta \, \mathcal{R}(u),
\end{equation*}
for $\alpha,\beta > 0$. The structure of $\mathcal{I}$ relies on the competition between the \emph{fidelity term} $\|Ku - f\|_H$, for a given normed space $H$ and a forward operator $K$, intended to keep track of the original datum $f$, and a \emph{regularizing term} $\mathcal{R}(u)$, whose task is to filter out the noise and return a clean image. The strength of both effects is quantified by the two \emph{tuning parameters} $\alpha$ and $\beta$.
\\
In this paper we propose a novel regularizer based on a Musielak-Orlicz generalization of classical Total Generalized Variation functionals, and provide a first study of the associated space of functions with bounded deformation with generalized Orlicz growth.
\\
Depending on specific applications, various types of variational regularizers have been proposed in the literature. A classical example is the ROF model, where $\mathcal{R}(u)$ is the total variation of $u$, cf. \cite{ROF92}. A known drawback of such functional is the so-called staircasing effect, namely the tendency of producing artificially piecewise-affine artifacts in the image denoising, see \cite{CL, J16}. In order to counteract such problems, several higher-order models have been introduced. A particularly successful one is the second-order Total Generalized Variation, defined as
\begin{equation*}
% \label{eq:classical-TGV}
    TGV^{2}_\alpha(u) := \sup \left\{ \int_\Omega u \, \div^2 \psi \, dx \st \psi \in C^2_c(\Omega; \R^{n \times n}_\sym), \|\psi\|_\infty \le \alpha_1, \|\div \, \psi\|_\infty \le \alpha_2 \right\},
\end{equation*}
for every $u \in L^1_\loc(\Omega)$, with $\alpha = (\alpha_1, \alpha_2) \in \R^+ \times \R^+$. We refer to the seminal paper \cite{BKP10} for a general introduction and higher-order generalizations, and to \cite{BV20} for an in-depth study of the second-order formulation.
\\
In the last few years, due to an increasing interest in accurate and adaptive image reconstruction, an emerging trend has been considering anisotropic models, allowing to tailor the denoising effects to specific features of different parts of the images. First steps in this direction have already been taken, e.g., in \cite{Bl, CL, SC}. In particular, in \cite{CLR} the authors have proposed an anisotropic ROF-type functional with $\mathcal{R}$ of the form $\mathcal{R}(u) = \int_\Omega \phi(x, |Du|) \, dx$ with
\begin{equation*}
    \phi(x, t) :=
    \begin{dcases}
        \tfrac{1}{p(x)} \, t^{p(x)} & \textnormal{ if } t \in [0,1], \\
        t - 1 + \tfrac{1}{p(x)} & \textnormal{ if } t > 1,
    \end{dcases}
\end{equation*}
and with the variable exponent $p : \Omega \to [c,2]$, with $c > 1$.
\\
In parallel, the study of variational problems in variable, Orlicz, and generalized Orlicz spaces has boomed in the past ten years. Among the vast literature, we single out the contributions \cite{Iwona_book,HHbook,HaH, HHL, HHLT, rao-ren} and the references therein.
\\
The extension of the model in \cite{CLR} to incorporate the case $c = 1$, as well as an Orlicz generalization for regularizers $\mathcal{R}(u) \simeq \int_\Omega \vphi(x,|Du|)$ for $\Phi$-functions $\vphi$ possibly having linear growth, is the subject of \cite{EHH}, where a thorough mathematical foundation of the space of anisotropic functions with bounded variation with generalized Orlicz growth has also been provided. A further step in this direction has been taken in \cite{GLM25}, where the associated Euler-Lagrange equations and gradient flows have also been analyzed (see also \cite{GLM25bis}). The effect of an anisotropic norm in the Chan-Vese (or anisotropic Mumford-Shah) is the subject of \cite{MP}. Anisotropic models, at least in the ROF setting, have indeed proven to allow for more accurate computations, as well as for a staircasing reduction and a sharper texture recognition. We refer to \cite{CWSL,EO,LMM} for an overview and further references.
\\
Our contribution is twofold. First, we introduce a novel anisotropic functional setting of functions with bounded deformation, namely the space $BD^\vphi$ of bounded deformation fields with generalized Orlicz growth. After analyzing its basic properties, we establish a link to a modular representation and provide a further study in the special case of variable Lebesgue integrability. 
\\
Second, we analyze a Musielak-Orlicz anisotropic Total Generalized Variation class of regularizers. Let $\Omega$ be a bounded, connected and open set in $\R^n$, with $n \ge 2$. Inspired by \cite{BKP10}, for a fixed generalized weak $\Phi-$function $\vphi \in \Phi_w(\Omega)$ (see Section \ref{orlicz} for the detailed assumptions), we define our anisotropic Total Generalized Variation of order $2$ with weight $\alpha = (\alpha_1, \alpha_2) \in \R^+ \times \R^+$ as
\begin{equation}
\label{TGV}
\begin{aligned}
    TGV^{\vphi,2}_\alpha(u):= \sup\bigg\{ &\int_\Omega u \,\div^2 \psi \, dx \st \\
    & \ \psi \in C^2_c (\Omega; \R^{n \times n}_\sym), \|\psi\|_{\vphi^*} \le \alpha_1, \|\div \, \psi\|_{\vphi^*} \le \alpha_2 \bigg\},
\end{aligned}
\end{equation}
for every $u \in L^1_{\textrm{loc}}(\Omega)$, where $\|\cdot\|_{\vphi^*}$ is the norm associated to the Fenchel conjugate function of $\vphi$, cf. Definition \ref{def:dual}, and where the divergence of order $2$ of $\xi \in \R^{n \times n}_\sym$ is defined as
\begin{equation*}
    \div^2 \xi := \sum_{i = 1}^n \frac{\partial^2 \xi_{ii}}{\partial x_i^2} + 2 \sum_{i<j} \frac{\partial^2 \xi_{ij}}{\partial x_i \partial x_j},
\end{equation*}
(see also \cite{BKP10} for higher-order notions of divergence). Note that, in principle, $TGV^{\varphi,2}(u)$ might take infinite values.

A prototypical example of admissible functions $\vphi$ is given by $\varphi(t) := \tfrac{1}{p(x)} t^{p(x)}$ (see also \cite[Example 4.3]{EHH}), with $p : \Omega \to [1,+\infty]$, thus possibly describing linear, superlinear, or even infinite growth in different portions of the domain.
\\
After studying the basic properties of $TGV^{\vphi,2}_\alpha$, we identify a dual representation, and show how classical existence and stability for classical Total Generalized Variation carry over to our Musielak-Orlicz setting.
\\
The paper is organized as follows. In Section \ref{prelim} we collect some definitions and technical results on $\Phi$-functions and generalized Orlicz spaces. Section \ref{sec:funct_anal} contains the definition and main properties of the novel space of  bounded deformation fields with generalized Orlicz growth $BD^\vphi$. In particular, in Subsection \ref{sec:prop_var_mod} we establish a connection between $BD^{\vphi}$ and an associated weighted-modular representation, while in Subsection \ref{sec:modular} we further establish a decomposition of the modular function in terms of the modular of the absolutely continuous part of the symmetric gradient and a singular part weighted by the recession function, cf. Theorem \ref{thm:exactFormula}. Eventually, in Subsection \ref{sec:BD_px} we specify our setting to the variable exponent case. The focus of Section \ref{sec:imaging} is on our anisotropic Total Generalized Variation. In Subsection \ref{sec:tgv_prop} we establish its main properties, and in Subsection \ref{sec:dual_formulations} we identify its dual reformulation in terms of a notion of anisotropic measure variation. Finally, in Subsection \ref{sec:ex_stab} we show stability with respect to the data and well-posedness of the minimization problem. In Appendices \ref{app:proof_mod} and \ref{AmarLemma} we give the proofs of some auxiliary results we need for the proof of Theorem \ref{thm:exactFormula}.

%%%%%%%%%%%%%%%%%%%%%%%%%%%%%%%%%%%%%%%%%%%%%%%%%%%

\section{Preliminary results}
\label{prelim}

In this section, after giving some notation in Subsection \ref{notation}, we collect some definitions and technical results on modulars and norms in Subsection \ref{mod_norm} and on generalized $\Phi$-functions and Orlicz spaces in Subsection \ref{orlicz}.

%-----------------------

\subsection{Notation}
\label{notation}

The notation $f \approx g$ means that there exist two constants $c_1, c_2 > 0$ such that $c_1 f \le  g \le c_2 f$. Given an exponent $p \in (1,+\infty)$, we denote by $p'$ its H\"older conjugate exponent, namely $p':=\frac{p}{p-1}$. If $p = 1$, then its H\"older conjugate exponent is $p' = +\infty$ and viceversa. In the sequel $\Omega$ denotes an open subset of $\mathbb R^n$. We denote by $L^0(\Omega)$ the set of measurable functions $f : \Omega \to \R$ and we use the notation $\mathcal{C}^k_c(\Omega; \R^n) := \{ \sigma \in \mathcal{C}^k(\Omega; \R^n) \st \supp(\sigma) \textnormal{ is compact in } \Omega \}$ with $k \in \N \cup \{+\infty\}$. We denote with $|\cdot|$ the Frobenius norm for matrices, i.e. $|M|^2 := \tr(M M^\T)$ where $M \in \R^{n \times n}$, and with $:$ the scalar product associated with the Frobenius norm, i.e. $A : B := \tr(A B^\T)$ for $A,B \in \R^{n \times n}$. With the same notation, namely $|\cdot|$, we also indicate the standard vector Euclidean norm, and with $\cdot$ the scalar product in $\R^n$.
\\
We say that a function $\omega: [0, +\infty) \to [0, +\infty]$ is a \textit{modulus of continuity} if it is increasing and such that $\omega(0) = \lim_{t \to 0^+} \omega(t)=0$. Note that we do not require concavity and we allow extended real values.
\\
We denote by $\R^{n \times n}_\sym$ the space of real-valued $n \times n$ symmetric matrices. By $\M(\Omega;X)$ we denote the space of all finite Radon measures in $\Omega$ with values in $X$, which can be $\R$, $\R^n$ or $\R^{n \times n}_\sym$ depending on the situation. Given a set $O \subset X$, we denote by $I_O$ the indicator function of the set $O$, namely
\begin{equation*}
    I_O(x) :=
    \begin{dcases}
        0 & \textnormal{ if } x \in O, \\
        + \infty & \textnormal{ if } x \notin O,
    \end{dcases}
\end{equation*}
and with $\chi_O$ the characteristic function of the set $O$, namely
\begin{equation*}
    \chi_O(x) :=
    \begin{dcases}
        1 & \textnormal{ if } x \in O, \\
        0 & \textnormal{ if } x \notin O.
    \end{dcases}
\end{equation*}
Given a function $u : \Omega \to \R^m$, we denote by $Du$ its distributional gradient. We indicate by $BV(\Omega;\R^m)$ the space of functions $u : \Omega \to \R^m$ such that $Du \in \M(\Omega;\R^m)$. If $m=1$ we denote the latter space simply by $BV(\Omega)$. When $n=m$, for any function $u : \Omega \to \R^n$ we use the notation $E u$ for the distributional symmetric gradient of $u$, namely $E u = \frac{1}{2} (Du + Du ^ \T)$, where $Du^\T$ stands for the transpose of $Du$. We write
\begin{equation*}
    BD(\Omega) := \{u \in L^1(\Omega;\R^n) \st E u \in \M(\Omega;\R^{n \times n}_\sym) \}.
\end{equation*}
We will often indicate by $c$ a generic constant, whose value may change from line to line.

%-----------------------

\subsection{Modulars and norms}
\label{mod_norm}

In order to define the spaces we use in our work, we need to give different notions of norms and modulars that generate our spaces. Note that our terminology differs from the original one in \cite{Mus83}, but follows the lines of \cite{EHH, HJR} in order to get a clearer connection between the modulars and their associated norms.

\begin{definition}
\label{def:norms}
    Let $X$ be a real vector space and let $\|\cdot\|: X \to [0, +\infty]$ be a function on $X$. Consider the following conditions:
    \begin{itemize}
        \item[(N1)] $\|f\| = 0$ if and only if $f = 0$,
        \item[(N2)] $\|af\| = |a| \, \|f\|$ for all $f \in X$ and $a \in \R$,
        \item[(N3)] $\|f + g\| \le \|f\| + \|g\|$ for all $f, g \in X$,
        \item[(N3')] there exists $c > 0$ such that $\|f + g\| \le c \, (\|f\| + \|g\|)$ for all $f,g \in X$.
    \end{itemize}
    We say that $\|\cdot\| : X \to [0, +\infty]$ is
    \begin{itemize}
        \item a quasi-seminorm if and only if it satisfies (N2) and (N3'),
        \item a seminorm if and only if it satisfies (N2) and (N3),
        \item a quasinorm if and only if it satisfies (N1), (N2) and (N3'),
        \item a norm if and only if it satisfies (N1), (N2) and (N3).
    \end{itemize}
\end{definition}

Note that conditions (N2) and (N3), equivalently (N2) and (N3'), already imply that if $f = 0_X$ then $\|f\| = 0$, namely condition (N1) only adds one of the two implications.

\begin{definition}
\label{def:modular}
    Let $X$ be a real vector space. A function $\vrho : X \to [0,+\infty]$ is called a quasi-semimodular on $X$ if and only if
    \begin{itemize}
        \item[(i)] $\vrho(0_X) = 0$,
        \item[(ii)] the function $\lambda \mapsto \vrho(\lambda x)$ is non-decreasing on $[0,+\infty)$ for every $x \in X$,
        \item[(iii)] $\vrho(-x) = \vrho(x)$ for every $x \in X$,
        \item[(iv)] there exists $\beta \in (0,1]$ such that for every $x,y \in X$ and every $\theta \in [0,1]$
        \begin{equation*}
            \vrho\left(\beta(\theta x + (1-\theta) y) \right) \le \theta \vrho(x) + (1-\theta) \vrho(y).
        \end{equation*}
    \end{itemize}
    If (iv) holds with $\beta = 1$, then $\vrho$ is called semimodular. A quasi-semimodular is called a quasimodular when $\vrho(x) = 0$ if and only if $x = 0_X$. A modular is defined analogously from a semimodular.
\end{definition}

\begin{definition}
\label{def_modular_space}
    Let $X$ be a real vector space. If $\vrho$ is a quasi-semimodular in $X$, then the modular space
    \begin{equation*}
        X_\vrho := \{x \in X \st \|x\|_\vrho < +\infty\}
    \end{equation*}
    is defined by the quasi-seminorm
    \begin{equation*}
        \|x\|_\vrho := \inf \left\{\lambda > 0  \st \vrho \left(\frac{x}{\lambda}\right) \le 1 \right\}.
    \end{equation*}
\end{definition}

%-----------------------

\subsection{Generalized \texorpdfstring{$\Phi$}{}-functions and Orlicz spaces}
\label{orlicz}

In this subsection we give some definitions regarding $\Phi$-functions and we then define generalized Orlicz (or Musielak-Orlicz) spaces, following the lines of \cite{HaH}. Since we are also going to make use of singular measures, the assumptions always need to get adjusted to hold for every point, as in \cite{HJR}.

\begin{definition}
    A function $f : \Omega \to \R$ is called $L$-almost increasing if there exists $L \ge 1$ such that $f(s) \le L f(t)$ for every $s \le t$. If this inequality holds for $L = 1$, the function is called non-decreasing. We define analogously $L$-almost decreasing and non-increasing functions.
\end{definition}

\begin{definition}
\label{def:phi-functions}
    We say that a function $\vphi: \Omega \times [0, +\infty) \to [0, +\infty]$ is a weak $\Phi$-function, and we write $\vphi \in \Phi_w(\Omega)$, if the following conditions hold:
    \begin{itemize}
        \item the function $x \mapsto \vphi(x,|f(x)|)$ is measurable for every function $f \in L^0(\Omega)$,
        \item the function $t \mapsto \vphi(x,t)$ is non-decreasing for every $x \in \Omega$,
        \item $\vphi(x,0) = \underset{t \to 0^+}{\lim} \vphi(x,t) = 0$ and $\underset{t \to +\infty}{\lim} \vphi(x,t) = +\infty$ for every $x \in \Omega$,
        \item the function $t \mapsto \tfrac{\vphi(x, t)}{t}$ is $L$-almost increasing on $(0,+\infty)$ for every $x \in \Omega$ with a constant $L$ independent of $x$.
    \end{itemize}
    If, moreover, $\vphi \in \Phi_w(\Omega)$ is convex and left-continuous with respect to $t$ for every $x \in \Omega$, then $\vphi$ is a convex $\Phi$-function, and we write $\vphi \in \Phi_c(\Omega)$. If $\vphi$ does not depend on $x$, then we omit the set and write $\vphi \in \Phi_w$ or $\vphi \in \Phi_c$.
\end{definition}

Note that since we deal with conjugates of functions possibly with linear growth at infinity, it is needed to allow for extended real-valued $\Phi$-functions. Now we are in the position of defining generalized Orlicz spaces.

\begin{definition}
% \label{def:spazi_Orlicz}
    Let $\vphi \in \Phi_w(\Omega)$ and define
    \begin{equation*}
    % \label{def:mod_phi}
        \vrho_\vphi(f) := \int_\Omega \vphi (x,|f|) \, dx \qquad \textnormal{for every } f \in L^0(\Omega).
    \end{equation*}
    Then, the set
    \begin{equation*}
    L^\vphi(\Omega) := (L^0(\Omega))_{\vrho_\vphi} = \left\{ f \in L^0(\Omega) \st \vrho_\vphi(\lambda  f) < +\infty \quad \textnormal{ for some } \lambda > 0 \right\}
    \end{equation*}
    is called a \textit{generalized Orlicz space} with quasinorm given by $\|f\|_\vphi := \|f\|_{\vrho_\vphi}$. We use the abbreviation $\|f\|_\vphi := \left\| |f| \right\|_\vphi$ for vector-valued functions.
\end{definition}

Thanks to \cite[Lemma 3.2.2]{HaH}, it is known that $\|\cdot\|_\vphi$ is a quasinorm in $L^\vphi(\Omega)$ if $\vphi \in \Phi_w(\Omega)$, and a norm if $\vphi \in \Phi_c(\Omega)$.
\\
Now we need to define some standard conditions on $\Phi$-functions that grant them additional properties, which sometimes are inherited by their conjugate functions as well.

\begin{definition}
\label{prop_Orlicz}
    Let $\vphi : \Omega \times [0,+\infty) \to [0,+\infty]$ and let $p,q > 0$. Then we define the following conditions:
    \begin{itemize}
        \item[(A0)] There exists $\beta \in (0,1]$ such that $\vphi(x,\beta) \le 1 \le \vphi(x, \tfrac{1}{\beta})$ for every $x \in \Omega$,
        
        \item[(A1)] For every $K > 0$ there exists $\beta \in (0,1]$ such that for every $x,y \in \Omega$
        \begin{equation*}
            \vphi(x,\beta t) \le \vphi(y,t) + 1 \quad \textnormal{when } \vphi(y, t) \in \left[0, \tfrac{K}{|x-y|^n}\right],
        \end{equation*}
        
        \item[(VA1)] For every $K > 0$ there exists a modulus of continuity $\omega$ such that for every $x,y \in \Omega$
        \begin{equation*}
            \vphi \left(x,\frac{t}{1 + \omega(|x-y|)}\right) \le \vphi(y,t) + \omega(|x-y|) \quad \textnormal{when } \vphi(y, t) \in \left[0, \frac{K}{|x-y|^n}\right],
        \end{equation*}
        
        \item[(aInc)$_p$] There exists $L_p \ge 1$ such that $t \mapsto \tfrac{\vphi(x,t)}{t^p}$ is $L_p$-almost increasing on $(0,+\infty)$ for every $x \in \Omega$,
        
        \item[(aDec)$_q$] There exists $L_q \ge 1$ such that $t \mapsto \tfrac{\vphi(x,t)}{t^q}$ is $L_q$-almost decreasing on $(0,+\infty)$ for every $x \in \Omega$.
    \end{itemize}
    We say that (aInc), respectively (aDec), holds if (aInc)$_p$, respectively (aDec)$_q$, holds for some $p > 1$, respectively $q > 1$.
\end{definition}

\begin{remark}
\label{rmk:prop_ainc}
    Note that, by Definition \ref{def:phi-functions}, every function $\vphi \in \Phi_w(\Omega)$ satisfies (aInc)$_1$.
\end{remark}

Note that if $\vphi \in \Phi_w(\Omega)$, then $\vphi(\, \cdot \, ,1) \approx 1$ implies (A0). Moreover, if $\vphi$ satisfies (aDec), then (A0) and $\vphi(\,\cdot\, ,1) \approx 1$ are equivalent. Assumption (A1) is an almost continuity condition and (aInc) and (aDec) are quantitative versions of the $\nabla_2$ and $\Delta_2$ conditions and measure lower and upper growth rates. We define the condition (A1) as in \cite{EHH}. We point out that it slightly differs from \cite{HaH, H15}, where it is assumed that 
\begin{equation*}
    \vphi(x,\beta t) \le \vphi(y,t) \quad \textnormal{when } \vphi(y, t) \in \left[1, \tfrac{1}{|B|}\right]
\end{equation*}
and $x$ and $y$ belong to the ball $B$. Lastly, the “vanishing (A1)” condition (VA1) is a continuity condition for $\vphi$, introduced in \cite{HOk} to prove maximal regularity of minimizers.
\\
In order to state the last property which we need for our analysis, we first need to define the recession function of $\vphi$.

\begin{definition}
\label{def:recession}
    Let $\vphi \in \Phi_w(\Omega)$. We call recession function of $\vphi$ the function $\vphi^\infty : \Omega \to [0,+\infty]$ defined by
    \begin{equation*}
        \vphi^\infty(x) := \limsup_{t \to +\infty} \frac{\vphi(x,t)}{t} \qquad \textnormal{for all } x \in \Omega.
    \end{equation*}
\end{definition}

Note that, since the continuity of $\vphi$ is not necessarily inherited by $\vphi^\infty$, this makes the non-autonomous case not trivial. For our analysis, we also need the following weaker version of (VA1) where at least one of the points has to belong to the set $\{\vphi^\infty < +\infty\}$.

\begin{definition}\label{restrVA1def}
    Let $\vphi \in \Phi_w(\Omega)$. We say that $\vphi$ satisfies restricted (VA1) if and only if it satisfies (A1) and for every $K > 0$ there exists $\omega$ modulus of continuity such that
    \begin{equation*}
            \vphi \left(x, \frac{t}{1 + \omega(|x-y|)}\right) \le \vphi(y,t) + \omega(|x-y|) \quad \textnormal{when } \vphi(y, t) \in \left[0, \frac{K}{|x-y|^n}\right],
        \end{equation*}
        for every $x,y \in \Omega$ such that $\vphi^\infty(x) < +\infty$ or $\vphi^\infty(y) < +\infty$.
\end{definition}

We now need to define the associate space of the generalized Orlicz space $L^\vphi(\Omega)$. In general, the associate space is a variant of the dual function space which works better at the end-points $p = 1$ and $p = \infty$. 
\begin{definition}
\label{def:dual}
    Let $\vphi \in \Phi_w(\Omega)$. We define the associate space $(L^\vphi)'(\Omega)=\{u \in L^0(\Omega):  \|u\|_{(L^\vphi)'}<+\infty\}$ with
    \begin{equation*}
        \|u\|_{(L^\vphi)'} := \sup_{\|v\|_\vphi \le 1} \int_\Omega uv \, dx.
    \end{equation*}
\end{definition}

Thanks to \cite[Theorem 3.4.6]{HaH}, it is well know that for weak $\Phi$-functions the associate space corresponds to the generalized Orlicz space $L^{\vphi^*}(\Omega)$, where
\begin{equation*}
    \vphi^*(x, t) := \sup_{s \ge 0} [st - \vphi(x, s)]
\end{equation*}
is the (Fenchel) conjugate function of $\vphi \in \Phi_w(\Omega)$. Moreover, it can be shown that (see \cite[Theorem 3.4.6]{EHH}) for all $u \in L^0(\Omega)$, $\|u\|_{L^{\varphi}}$  is equivalent to $\sup_{\|v\|_{L^{\vphi^*}} \le 1} \int_\Omega uv \, dx.$

Let us state some well-known properties of the conjugate function.

\begin{lemma}
\label{lmm:prop_conj}
    Let $\vphi \in \Phi_w(\Omega)$. Then,
    \begin{itemize}
        \item[(i)] $\vphi^* \in \Phi_c(\Omega)$, meaning that $\vphi^*$ is always convex and left-continuous,
        
        \item[(ii)] If $p,q \in (1,+\infty)$, then $\vphi$ satisfies (aInc$_p$) or (aDec$_q$) if and only if $\vphi^*$ satisfies respectively (aDec$_{p'}$) or (aInc$_{q'}$),
        
        \item[(iii)] If $\vphi \in \Phi_c(\Omega)$, then
        \begin{equation*}
            \vphi^*\left(x, \tfrac{\vphi(x, t)}{t}\right) \le \vphi(x, t) \quad \textnormal{for every } x \in \Omega \textnormal{ and } t \in \R^+,
        \end{equation*}
        and $\vphi^{**} = \vphi$,
        
        \item[(iv)] If $\vphi$ satisfies (A0) or (A1), then so does $\vphi^*$.
    \end{itemize}
    \qed
\end{lemma}

Conditions (i), (ii) and (iv) have been proven in \cite{HaH}, respectively in Lemma 2.4.1, Proposition 2.4.9 and Lemmas 3.7.6 and 4.1.7, while condition (iii) has been shown in \cite[Corollary 2.6.3]{HHbook}, see also \cite[Proposition 2.4.5]{HaH}.
\\
Lastly, we report the following result regarding some immersions of generalized Orlicz spaces into Lebesgue spaces.

\begin{lemma}
\label{lmm:inclusions}
    Let $\Omega \subset \R^n$ be such that $|\Omega| < +\infty$ and let $\vphi \in \Phi_w(\Omega)$. If $\vphi$ satisfies (A0), then
    \begin{equation*}
        L^\infty(\Omega) \hookrightarrow L^\vphi(\Omega) \hookrightarrow L^1(\Omega).
    \end{equation*}
    Moreover, if $\vphi$ satisfies (aInc)$_p$, then $L^\vphi(\Omega) \hookrightarrow L^p(\Omega)$.
    \qed
\end{lemma}

Note that the second continuous inclusion comes from the fact that every weak $\Phi$-function satisfies (aInc)$_1$, see Remark \ref{rmk:prop_ainc}. This Lemma was shown in \cite[Corollary 3.7.9 \& Corollary 3.7.10]{HaH}. We point out that condition (A0) is essential both for establishing the above embeddings and to ensure that $L^\vphi(\Omega)$ is a Banach space, cf. \cite[Theorem 3.7.13]{HaH}. In particular, of (A0) is violated, the first inclusion in Lemma \ref{lmm:inclusions} might fail to hold. We refer to \cite[Example 3.7.11]{HaH} for a counterexample.

%%%%%%%%%%%%%%%%%%%%%%%%%%%%%%%%%%%%%%%%%%%%%%%%%%%

\section{The space of bounded deformations fields with generalized Orlicz growth}
\label{sec:funct_anal}

In \cite[Definition 4.1]{EHH}, given a function $\vphi \in \Phi_w(\Omega)$, an anisotropic total variation (there called {\it dual norm}) is introduced exploiting the conjugate function of $\vphi$, namely for any $u \in L^1(\Omega)$ we have
\begin{equation}
\label{def:EHH_TV}
    V_\vphi(u) := \sup \left\{ \int_\Omega u \, \div \, w \, dx \st w \in C^1_c(\Omega; \R^n), \|w\|_{\vphi^*} \le 1 \right\}.
\end{equation}
Moreover, the space of functions of bounded variation with generalized Orlicz growth is defined as
\begin{equation}
\label{def:BV_phi}
    BV^\vphi(\Omega) := \{ u \in L^\vphi(\Omega) \st \|u\|_{BV^\vphi(\Omega)} < +\infty \},
\end{equation}
where
\begin{equation}
\label{norm_BVphi}
    \|u\|_{BV^\vphi(\Omega)} := \|u\|_\vphi + V_\vphi(u).
\end{equation}
Equivalently, given $\vphi \in \Phi_w(\Omega)$ in \cite{BHHZ} a vectorial anisotropic total variation is introduced, namely for any $m \in \N^+$ and $u \in L^1_\loc(\Omega;\R^m)$ we have
\begin{equation}
\label{var_BHHZ}
    V^m_\vphi(u) := \sup \left \{\sum_{i = 1}^m \int_\Omega u_i \, \div \, w_i \, dx \st w \in \left[C^1_c(\Omega; \R^n)\right]^m, \|w\|_{\vphi^*} \le 1 \right\}.
\end{equation}
One can equivalently write
\begin{equation*}
    V^m_\vphi(u) = \sup \left \{ \int_\Omega u \cdot \div \, w \, dx \st w \in C^1_c(\Omega; \R^{m \times n}), \|w\|_{\vphi^*} \le 1 \right\},
\end{equation*}
where the divergence $\div \, w$ is a vector in $\R^m$, namely taken a matrix-valued field $\xi$, it is computed row-wise, i.e.
\begin{equation}
\label{divrow}
    (\div \, \xi)_i :=  \div \, \xi_i = \sum_{j = 1}^n \frac{\partial \xi_{ij}}{\partial x_j},
\end{equation}
for any $i \in [1,m] \cap \N$ and any $\xi : \Omega \to \R^{m \times n}$. Moreover, we say that $u$ belongs to $BV^\vphi(\Omega;\R^m)$ if
\begin{equation*}
    \|u\|_{BV^\vphi(\Omega;\R^m)} := \|u\|_\vphi + V_\vphi^m(u) < +\infty.
\end{equation*}
In the case $m = 1$, \eqref{var_BHHZ} reduces to \eqref{def:EHH_TV} and we denote $BV^\vphi(\Omega;\R)$ simply by $BV^\vphi(\Omega)$.
\\
Note that in general Harjulehto \& H\"ast\"o in their works, see e.g. \cite{HaH}, use the symbol $C^1_0(\Omega; \R^n)$ to denote the space $C^1_c(\Omega; \R^n)$. For this reason, we gave the definition of the total variation with our notation instead of the one used in \cite{EHH}.

\begin{remark}
\label{rmk:A0_si_no}
    Clearly, when $\vphi(x,t) = |t|$, $BV^\varphi(\Omega;\R^m)= BV(\Omega;\R^m)$. A priori, without assuming any other property on $\vphi$ except for it to be in the space $\Phi_w(\Omega)$, the spaces $BV(\Omega)$ and $BV^\vphi(\Omega)$ are not related to each other. Considering $\Omega \subset \R^n$ bounded, one gets the continuous inclusion $BV^\vphi(\Omega) \hookrightarrow BV(\Omega)$ assuming (A0) on the function $\vphi$, see Lemma \ref{lmm:inclusions}. The fact that $\vphi$ satisfies (aInc)$_1$ is already included in the definition of weak $\Phi$-function, see Remark \ref{rmk:prop_ainc}.
\end{remark}

\begin{remark}
    In \cite{AB} and \cite{M05}, in the simple anisotropic case, (i.e. $\varphi(x,t)$ with linear behavior in $t$) the dual function $\vphi^0$ is used instead of the conjugate function $\vphi^*$ in order to define a weighted total variation (they call it {\it generalized total variation}). But under the assumptions considered in \cite{AB} and \cite{M05}, the two definitions coincide. Indeed, let us assume that $\vphi: \Omega \times [0, +\infty) \to [0,+\infty)$ is a continuous function which is convex in the second variable, satisfies (A0) and such that
    \begin{equation*}
        \vphi(x,t) = t \, \vphi(x,1) \quad \textnormal{ for any } x \in \Omega \textnormal{ and } t \in [0, +\infty).
    \end{equation*}
    Assume that there exist $0 < \lambda \le \Lambda < +\infty$ such that 
    \begin{equation*}
        \lambda t \le \vphi(x,t) \le \Lambda t \qquad \textnormal{ for any } x \in \Omega \textnormal{ and for any } t \in [0, +\infty).
    \end{equation*}
    Moreover, let us consider $\psi : \Omega \times \R^n \to [0, +\infty)$ defined as $\psi(x,\xi) := \vphi(x, |\xi|)$ for any $x \in \Omega$ and $\xi \in \R^n$. Note that $\psi$ is continuous, convex in the second variable, and such that
    \begin{equation*}
        \psi(x, t \xi) = \vphi(x, |t \xi|) = |t|\vphi (x, |\xi|) = |t| \psi(x,\xi)
    \end{equation*}
    and
    \begin{equation*}
        \lambda |\xi|  \le \psi(x, \xi) \le \Lambda |\xi|,
    \end{equation*}
    for every $t \in \R$, $\xi \in \R^n$ and every $x \in \Omega$.
    \\
    In \cite{AB}, considered $u \in BV(\Omega)$ and under these assumptions on $\psi$, the anisotropic total variation (there called {\it generalized total variation of $u$ with respect to $\psi$ in $\Omega$}) has been characterized, see in particular \cite[Proposition 3.2]{AB}, as
    \begin{equation}
    \label{Vpsieq}
        \overline{V_\psi}(u) := \sup \left\{\int_\Omega u \, \div w \, dx \st w \in C^1_c(\Omega;\R^n), \psi^0(x,w(x)) \le 1 \textnormal{ for all } x \in \Omega \right\}, 
    \end{equation}    
    where $\psi^0$ is the dual function of $\psi$, namely
    \begin{equation*}
        \psi^0(x, \xi^*) := \sup \{(\xi^*, \xi) \st \xi \in \R^n, \psi(x, \xi) \le 1\}.
    \end{equation*}
    Note that in \cite{AB} they use the notation $C^1_0(\Omega;\R^n)$ to indicate the space that we denote with $C^1_c(\Omega;\R^n)$. Then, given a function $u \in BV(\Omega)$, the anisotropic total variation of $u$ with respect to $\vphi$ in $\Omega$ introduced in \cite{EHH} as \eqref{def:EHH_TV}, coincides with the notion given by \cite{AB}, namely \eqref{Vpsieq}.
    \\
    Indeed, by \cite[(5.10)]{AB} the right-hand side of \eqref{Vpsieq} coincides with
    \begin{equation}
    \label{Vpsieq2}
        \sup\left\{\int_\Omega u \, \div \, w \, dx \st w \in C^1_c(\Omega;\R^n), \int_\Omega \psi^*(x, w(x)) \, dx < +\infty \right\}.
    \end{equation}
    Considering the even extension of $\vphi(x,\cdot)$, for any $x \in \Omega$, and since $\vphi(x,\cdot)$ is convex, in view of our definition of $\psi$ and Lemma \ref{lemmaApp}, which guarantees that
    \begin{equation}
    \label{psi*varphi*mod}
        \psi^*(x,\cdot)= \vphi^*(x,|\cdot|) \qquad \textnormal{ for a.e. } x \in \Omega,
    \end{equation}
    we have that \eqref{Vpsieq2} is equal to
    \begin{equation}
    \label{Vpsieq3}
        \sup \left\{\int_\Omega u \, \div \, w \, dx \st w \in C^1_c(\Omega;\R^n), \int_\Omega \vphi^*(x, |w(x)|) \, dx < +\infty \right\}.
    \end{equation}
    Taking into account \cite[(2.18)]{AB}, namely
    \begin{equation*}
        \psi^*(x, \xi^*) =
        \begin{dcases}
            0 & \textnormal{ if } \psi^0(x, \xi^*) \le 1, \\
            +\infty & \textnormal{ if } \psi^0(x, \xi^*) > 1,
        \end{dcases}
    \end{equation*}
    for every $x \in \Omega$ and every $\xi^* \in \R^n$, we have that \eqref{Vpsieq3} can be equivalently written as
    \begin{equation}
    \label{Vpsieq4}
        \sup \left\{\int_\Omega u \, \div \, w \, dx \st w \in C^1_c(\Omega;\R^n), \int_\Omega \vphi^*(x, |w(x)|) \, dx \le 1 \right\},
    \end{equation}
    where \eqref{psi*varphi*mod} has been exploited again. Now, using the definition
    \begin{equation*}
        \|w\|_{\vphi^*} := \inf \left\{\lambda > 0 \st \int_\Omega \vphi^* \left(x, \left|\frac{w(x)}{\lambda} \right| \right) \, dx < 1 \right\},
    \end{equation*}
    we get that \eqref{Vpsieq4} is equal to \eqref{def:EHH_TV}.
\end{remark}

In order to treat the minimization problem using $TGV^{\vphi,2}_\alpha$ as a regularizer, see \eqref{TGV}, we need to introduce a variant of the generalized variation associated to $\vphi$ defined in \eqref{def:EHH_TV}. For every $v \in \M(\Omega;\R^m)$ with $m \in \N^+$, we define
\begin{equation}
\label{V_tilde}
    \widetilde{V}^m_\vphi(v) := \sup \left\{\int_\Omega \psi \cdot dv(x) \st \psi \in C^1_c(\Omega;\R^m), \|\psi\|_{\vphi^*} \le 1 \right\}.
\end{equation}
If we are working on a set $A \subsetneq \Omega$, then we will denote it by $\widetilde{V}^m_\vphi(v;A)$.

\begin{remark}
\label{rk:v-equiv}
    Note that for any $u \in BV(\Omega)$, in order to know that at least that $Du \in \M(\Omega;\R^n)$, an integration by parts yields
    \begin{equation*}
        \widetilde{V}^n_\vphi(Du) = V_\vphi(u),
    \end{equation*}
    where the second term is defined as in \eqref{def:EHH_TV}, namely the anisotropic total variation introduced in \cite{EHH}.
\end{remark}

In the following, we will either use the above notation for $m = n$ or $m = n \times n$. Note that in the latter case the scalar product coincides with the one underlying the Frobenius norm in $\R^{n \times n}$, namely reading for any $v \in \M(\Omega;\R^{n\times n}_\sym)$ as 
\begin{equation*}
    \widetilde{V}^{n\times n}_\vphi(v) = \sup \left\{\int_\Omega \psi : dv(x) \st \psi \in C^1_c(\Omega;\R^{n \times n}_\sym), \|\psi\|_{\vphi^*} \le 1 \right\}.
\end{equation*}
Considering test functions in $C^1_c(\Omega;\R^{n \times n}_\sym)$ or $C^1_c(\Omega;\R^{n \times n})$ is equivalent when $v \in \M(\Omega;\R^{n \times n}_\sym)$, since the scalar product only takes into account the symmetric part of the test function. In general the notation $\widetilde{V}^{n\times n}_\vphi$ could be used also for measures in $\M(\Omega;\R^{n \times n})$ without any symmetry property, but that is not the case in this work. We avoid specifying the symmetry in order to lighten the notation.

\begin{remark}
\label{rmk:struttura}
    Considering an open set $\Omega \subset \R^n$, if $v$ has the special structure $v = Du - w$, with $u,w \in \M(\Omega;\R^n)$, an integration by parts entails
    \begin{equation*}
        \widetilde{V}^n_\vphi(Du - w) = \sup \left\{\int_\Omega u \, \div \, \psi \, dx + \int_\Omega \psi \cdot dw(x) \st \psi \in C^1_c(\Omega;\R^n), \|\psi\|_{\vphi^*} \le 1 \right\},
    \end{equation*}
    see also \cite{DFFI}. Note that it is possible to rewrite it like this by \cite[Exercise 3.2]{AFP}, indeed having $u \in \M(\Omega;\R^n)$ and $Du = v + w \in \M(\Omega;\R^n)$ implies that $u \in BV(\Omega)$.
    \\
    In the same way, if $\Omega$ is also bounded and $C^1$, by \cite[Chapter II, Theorem 2.3]{temam}, having $Eu \in \M(\Omega;\R^{n \times n}_\sym)$ implies $u \in BD(\Omega)$. Namely, in the special case $v = Eu \in \M(\Omega;\R^{n\times n}_\sym)$, \eqref{V_tilde} rewrites as
    \begin{equation}
    \label{var_grad_sym}
        \widetilde{V}^{n \times n}_\vphi(E u) = \sup \left\{\int_\Omega E^* \psi \cdot u \, dx \st \psi \in C^1_c(\Omega;\R^{n \times n}_\sym), \|\psi\|_{\vphi^*} \le 1 \right\},
    \end{equation}
    where $E^* \psi = -\div \, \psi$, and $\div$ denotes the matrix divergence of order $1$,  computed  row-wise (i.e. with respect to columns), cf. \eqref{divrow}.
\end{remark}

\begin{remark}
    We point out that, here and in the following, the set $\Omega$ is only required to have a $C^1$ boundary in order to apply \cite[Chapter II, Theorem 2.3]{temam}. The extension of the above result to the case of Lipschitz boundary will be the subject of the forthcoming paper \cite{DR25}.
\end{remark}

Thanks to the definition of generalized variation associated to $\vphi$ in \eqref{V_tilde}, we can now define the space of bounded deformation fields with generalized Orlicz growth, namely
\begin{equation}
\label{BD_phi}
    BD^\vphi(\Omega) := \left\{ w \in L^\vphi(\Omega;\R^n) \st \|w\|_{BD^\vphi(\Omega)} < +\infty \right\},
\end{equation}
where
\begin{equation*}
    \|w\|_{BD^\vphi(\Omega)} := \|w\|_\vphi + \widetilde{V}^{n \times n}_\vphi(Ew).
\end{equation*}

\begin{lemma}
\label{lmm:seminorm}
    Let $\Omega$ be a bounded and open set and let $\vphi \in \Phi_w(\Omega)$. Then $\widetilde{V}^{n\times n}_\vphi$ is a seminorm.
\end{lemma}
\begin{proof}
    We need to check that $\widetilde{V}^{n \times n}_\vphi$ satisfies conditions (N2) and (N3) in Definition \ref{def:norms}. The homogeneity property (N2) is obvious, in fact $\widetilde{V}^{n \times n}_\vphi(av) = |a|\widetilde{V}^{n \times n}_\vphi(v)$ for every $v \in \M(\Omega;\R^{n \times n}_\sym)$ and any $a \in \R$. In order to check (N3), namely the triangle inequality, let us consider $v,w \in \M(\Omega;\R^{n \times n}_\sym)$. We have
    \begin{equation*}
        \int_\Omega \psi : d(v+w) = \int_\Omega \psi : dv + \int_\Omega \psi : dw \le \widetilde{V}^{n \times n}_\vphi(v) + \widetilde{V}^{n \times n}_\vphi(w)
    \end{equation*}
    for any $\psi \in C^1_c(\Omega;\R^{n \times n}_\sym)$ such that $\|\psi\|_{\vphi^*} \le 1$. By taking the supremum over such $\psi$ we have the thesis.
\end{proof}

This result entails that if $\vphi \in \Phi_w(\Omega)$, then $\|\cdot\|_{BD^\vphi(\Omega)}$ is a quasinorm. In addition, if $\vphi$ verifies (A0), or if $\vphi \in \Phi_c(\Omega)$, then $\|\cdot\|_{BD^\vphi(\Omega)}$ is a norm, see \cite[Lemma 3.2.2 \& Theorem 3.7.13]{HaH}.

\begin{remark}
\label{rmk:A0_si_no_BD}
    As for $BV(\Omega)$ and $BV^\vphi(\Omega)$, see Remark \ref{rmk:A0_si_no}, without assuming any other property on $\vphi$ except for it to be in the space $\Phi_w(\Omega)$, the spaces $BD(\Omega)$ and $BD^\vphi(\Omega)$ are not related to each other. Considering a set $\Omega$ bounded, one gets the continuous immersion $BD^\vphi(\Omega) \hookrightarrow BD(\Omega)$ assuming (A0) on the function $\vphi$, see Lemma \ref{lmm:inclusions}.
\end{remark}

\begin{remark}
\label{rmk:BD-phi=BD-tilde}
    Since a priori we define the space starting from a function $\vphi \in \Phi_w(\Omega)$, one may ask whether the space defined starting from $\vphi^{**}$, always a convex $\Phi$-function, see Lemma \ref{lmm:prop_conj}(i), namely
    \begin{equation*}
        BD^{\vphi^{**}}(\Omega) := \left\{ w \in L^{\vphi^{**}}(\Omega;\R^n) \st \|w\|_{BD^{\vphi^{**}}(\Omega)} < +\infty \right\},
    \end{equation*}
    with
    \begin{equation*}
        \|w\|_{BD^{\vphi^{**}}(\Omega)} := \|w\|_{\vphi^{**}} + \widetilde{V}^{n \times n}_{\vphi^{**}}(E w),
    \end{equation*}
    differs from the space $BD^\vphi(\Omega)$. Thanks to \cite[Proposition 2.4.5 \& Proposition 3.2.4]{HaH} it is possible to prove that $BD^\vphi(\Omega) = BD^{\vphi^{**}}(\Omega)$, indeed the norms $\|\cdot\|_\vphi$ and $\|\cdot\|_{\vphi^{**}}$ are equivalent and the variations coincide, implying that the norms $\|\cdot\|_{BD^\vphi(\Omega)}$ and $\|\cdot\|_{BD^{\vphi^{**}}(\Omega)}$ are equivalent as well. Note that some results needed for this consideration, e.g. \cite[Lemma 2.4.3]{HaH} or \cite[Proposition 2.4.5]{HaH}, are only stated without the $x$-dependence, but they hold true even in the generalized Orlicz setting, while others, e.g. \cite[Theorem 2.5.10]{HaH} or \cite[Proposition 3.2.4]{HaH}, are already stated with the $x$-dependence.
\end{remark}

With the next example we want to show that, at least considering some $\Phi$-functions for which a Korn-type inequality does not hold, the space $BV^\vphi(\Omega)$ is a proper subset of $BD^\vphi(\Omega)$. The example follows the lines of \cite[Example 7.7]{ACD}, exploiting a construction developed in \cite{O62}, see also \cite[Chapter I, Subsection 1.3]{temam}.

\begin{example}
    We first remark that \eqref{var_BHHZ}, $BV^\vphi(\Omega;\R^n) \subset BD^\vphi(\Omega)$. We show below that, at least in some situations, the inclusion is strict.
    \\
    Indeed, without loss of generality we can consider an open set $\Omega \supset \supset B$, with $B$ a small ball centered at $0$, and a function such that $\vphi(x,t) = |t|$ for a.e. $x \in B$. Then, the same field $u$ provided by \cite[Example 7.7]{ACD}, which has compact support in $B$ and is in $BD(\Omega) \setminus BV(\Omega; \R^n)$ provides the desired example of a field in $BD^\vphi(\Omega) \setminus BV^\vphi(\Omega; \R^n)$.
    \\
    Note that many $\Phi$-functions, e.g. many out of the ones listed in \cite[Example 2.5.3]{HaH}, fulfill the assumption of coinciding with $|t|$ on a ball $B$, meaning that this example holds true in many contexts.
\end{example}

The validity of Korn-type inequalities in the context of Orlicz spaces has been studied in different contexts, for example \cite[Theorem 1.1]{BD12} where it was shown that the $\nabla_2$-condition and the $\Delta_2$-condition are sharp conditions for Korn-type inequalities in Orlicz spaces (for the definition of these two conditions see \cite[Subsection 2.3]{rao-ren} or \cite[Equations (2.6) \& (2.7)]{BCD17}). In \cite[Theorems 3.1 \& 3.3]{C14} the validity of Korn-type inequalities for Young functions was shown, while in \cite[Theorems 3.1 \& 3.3]{BCD17} the equivalence of these Korn-type inequalities to two balance conditions, similar to the $\nabla_2$- and $\Delta_2$-conditions, has been proven. Note that, under the assumptions considered in \cite{C14} and \cite{BCD17}, their definition of Young functions and our definition of convex $\Phi$-functions, when independent of $x$, are equivalent. In the context of generalized Orlicz spaces (or Musielak-Orlicz spaces) instead, there are results in the case of variable exponents Lebesgue spaces when the exponent is bounded away from $1$, see \cite[Section 5]{DR03} and \cite[Theorems 14.3.21 \& 14.3.23]{HHbook}, but not much else is known.

%--------------------------------------------

\subsection{Weighted modular and associated norms}
\label{sec:prop_var_mod}

Analogously to \cite[Definition 4.1]{EHH}, we define the anisotropic modular, or {\it dual modular}, associated to the anisotropic total variation $\widetilde{V}^{n \times n}_\vphi$, defined as in \eqref{V_tilde}. In particular, we define for any $v \in \M(\Omega; \R^{n \times n}_\sym)$
\begin{equation}
\label{varrhosymvarphi}
    \vrho_{\widetilde{V}^{n \times n}_\vphi}(v) := \sup \left\{ \int_\Omega \psi : dv(x) - \int_\Omega\vphi^*(x,|\psi|) \, dx \st \psi \in C^1_c(\Omega; \R^{n \times n}_\sym) \right\}.
\end{equation}

\begin{lemma}
\label{lmm:semimod}
    Let $\Omega$ be a bounded and open set and let $\vphi \in \Phi_w(\Omega)$. Then $\vrho_{\widetilde{V}^{n \times n}_\vphi}$ is a semimodular.
\end{lemma}
\begin{proof}
    We need to check that $\vrho_{\widetilde{V}^{n \times n}_\vphi}$ satisfies the conditions of Definition \ref{def:modular}. Conditions (i) and (iii) are obvious, respectively thanks to the properties of $\vphi^*$ and that if $\psi$ is an admissible test function, then also $-\psi$ is. Condition (iv) is also easy to prove using the linearity of the first integral in the definition of $\vrho_{\widetilde{V}^{n \times n}_\vphi}$ and the fact that the supremum of a sum is smaller than the sum of the suprema.
    \\
    We are only left to prove condition (ii), namely that the function $\lambda \mapsto \vrho_{\widetilde{V}^{n \times n}_\vphi}(\lambda v)$ on $[0, +\infty)$ is non-decreasing for any $v \in \M(\Omega; \R^{n \times n}_\sym)$. Indeed, if $\lambda = 0$, then $\vrho_{\widetilde{V}^{n \times n}_\vphi}(\lambda v) = 0$, but we know already that for any other $\mu > \lambda = 0$ then $\vrho_{\widetilde{V}^{n \times n}_\vphi}(\mu v) \ge 0$, since we can always choose $\psi = 0$ as a test function. Let us instead fix $0 < \lambda < \mu < +\infty$. We point out that since both $\psi$ and $-\psi$ are possible test functions, we can also only consider without loss of generality functions in the space
    \begin{equation*}
        \overline{T}^+ := \left\{ \psi \in C^1_c(\Omega; \R^{n \times n}_\sym) \st \int_\Omega \psi : dv(x) \ge 0 \right\},
    \end{equation*}
    since the second integral in the definition of the modular is always zero or negative. Then we have
    \begin{equation*}
    \begin{aligned}
        \vrho_{\widetilde{V}^{n \times n}_\vphi}(\lambda v)
        &= \sup \left\{ \int_\Omega \lambda\psi : dv(x) - \int_\Omega\vphi^*(x,|\psi|) \, dx \st \psi \in C^1_c(\Omega; \R^{n \times n}_\sym) \right\} \\
        &= \sup \left\{ \lambda \int_\Omega \psi : dv(x) - \int_\Omega\vphi^*(x,|\psi|) \, dx \st \psi \in \overline{T}^+ \right\} \\
        &\le \sup \left\{ \mu \int_\Omega \psi : dv(x) - \int_\Omega\vphi^*(x,|\psi|) \, dx \st \psi \in \overline{T}^+ \right\} \\
        &= \sup \left\{\int_\Omega \mu\psi : dv(x) - \int_\Omega\vphi^*(x,|\psi|) \, dx \st \psi \in C^1_c(\Omega; \R^{n \times n}_\sym) \right\} = \vrho_{\widetilde{V}^{n \times n}_\vphi}(\mu v),
    \end{aligned}
    \end{equation*}
    for any $v \in \M(\Omega; \R^{n \times n}_\sym)$, proving condition (ii).
    \\
    Since condition (iv) holds with the constant $\beta = 1$, then $\vrho_{\widetilde{V}^{n \times n}_\vphi}$ is a semimodular.
\end{proof}

Note that, analogously to the anisotropic total variation, see Remark \ref{rmk:struttura}, considering a bounded, $C^1$ and open set $\Omega \subset \R^n$, in the special case $v = Eu \in \M(\Omega; \R^{n \times n}_\sym)$ this rewrites as
\begin{equation}
\label{varrhosymvarphi_Eu}
    \vrho_{\widetilde{V}^{n \times n}_\vphi}(Eu) = \sup \left\{ \int_\Omega \left[u \cdot E^*\psi - \vphi^*(x,|\psi|) \right] dx \st \psi \in C^1_c(\Omega; \R^{n \times n}_\sym) \right\},
\end{equation}
where $E^*\psi = - \div \, \psi$.

\begin{remark}
\label{rmk:test_mod}
    On the one hand, note that testing \eqref{varrhosymvarphi_Eu} with $\psi = 0$, we get that the supremum in $\vrho_{\widetilde{V}^{n \times n}_\vphi}$ is always non-negative. On the other hand, if $\vrho_{\vphi^*}(|\psi|) = +\infty$, then the integral in \eqref{varrhosymvarphi_Eu} is equal to $-\infty$, since the first term is finite. Here we used again  \cite[Chapter II, Theorem 2.3]{temam} yielding that $Eu \in \M(\Omega;\R^{n\times n}_\sym)$ implies $u \in BD(\Omega)$.

    This means that test functions $\psi$ with $\vrho_{\vphi^*}(|\psi|) = +\infty$ can be omitted and $\vrho_{\widetilde{V}^{n \times n}_\vphi}$ can be equivalently rewritten as
    \begin{equation*}
        \vrho_{\widetilde{V}^{n \times n}_\vphi}(Eu) := \sup \left\{ \int_\Omega \left[u \cdot E^*\psi - \vphi^*(x,|\psi|) \right] dx \st \psi \in C^1_c(\Omega; \R^{n \times n}_\sym), \vrho_{\vphi^*}(|\psi|) < +\infty \right\}.
    \end{equation*}
    The analogous reasoning can be repeated for the general expression of $\vrho_{\widetilde{V}^{n \times n}_\vphi}$ in \eqref{varrhosymvarphi}.
\end{remark}

In Lemma \ref{lmm:semimod} we showed that $\vrho_{\widetilde{V}^{n \times n}_\vphi}$ is a semimodular, implying that, thanks to the Luxemburg method, we obtain a seminorm $\|\cdot\|_{\vrho_{\widetilde{V}^{n \times n}_\vphi}}$. The following result relates the anisotropic variation we introduced in \eqref{V_tilde} with the seminorm associated to the modular defined in \eqref{varrhosymvarphi}, similarly to \cite[Lemma 4.7]{EHH}.

\begin{lemma}
\label{lem:equivalence}
    Let $\Omega \subset \R^n$ be a bounded, connected, $C^1$ and open set. Let $\vphi \in \Phi_w(\Omega)$ and let $u \in BD^\vphi(\Omega)$. Then
    \begin{equation*}
        \|Eu\|_{\vrho_{\widetilde{V}^{n \times n}_\vphi}} \le \widetilde{V}^{n\times n}_\vphi(Eu) \le 2 \, \|Eu\|_{\vrho_{\widetilde{V}^{n \times n}_\vphi}}.
    \end{equation*}
\end{lemma}
\begin{proof}
    Since $u \in BD^\vphi(\Omega)$, the case $\widetilde{V}^{n \times n}_\vphi(E u) = +\infty$ is excluded a priori. Moreover, if $\widetilde{V}^{n \times n}_\vphi(Eu) = 0$, then $\vrho_{\widetilde{V}^{n \times n}_\vphi}(\frac{Eu}{\lambda}) = 0$ for every $\lambda > 0$, yielding $\|Eu\|_{\vrho_{\widetilde{V}^{n \times n}_\vphi}} = 0$. Lastly, since the claim is homogeneous, the case $\widetilde{V}^{n \times n}_\vphi(Eu) \in (0,+\infty)$ can always be reduced to $\widetilde{V}^{n \times n}_\vphi(Eu) = 1$. 
    \\
    By the definition of $\widetilde{V}^{n \times n}_\vphi$, see \eqref{var_grad_sym}, for any $w \in C^1_c(\Omega;\R^{n \times n}_\sym)$ such that $\|w\|_{\vphi^*} \le 1$ it follows that 
    \begin{equation*}
        \int_\Omega u \cdot E^*w \, dx \le \widetilde{V}^{n \times n}_\vphi(Eu) \, \|w\|_{\vphi^*} = \|w\|_{\vphi^*} \le 1 + \vrho_{\vphi^*}(|w|),
    \end{equation*}
    where the last inequality comes from a general property of the Luxemburg norms, see \cite[Corollary 2.1.15]{HHbook}. Thus, by the definition of $\vrho_{\widetilde{V}^{n \times n}_\vphi}$ given in \eqref{varrhosymvarphi},
    \begin{equation*}
        \vrho_{\widetilde{V}^{n \times n}_\vphi}(Eu) \le \sup_{w \in C^1_c(\Omega; \R^{n \times n}_\sym)} [1 + \vrho_{\vphi^*}(|w|) - \vrho_{\vphi^*}(|w|)] = 1,
    \end{equation*}
    implying $\|E u\|_{\vrho_{\widetilde{V}^{n \times n}_\vphi}} \le 1$. This concludes the proof of the first inequality. 
    \\
    We next establish the opposite inequality $2 \, \|Eu\|_{\vrho_{\widetilde{V}^{n \times n}_\vphi}} \ge 1$, which is equivalent to
    \begin{equation}
    \label{eq:mod_magg}
        \vrho_{\widetilde{V}^{n \times n}_\vphi} \left(\frac{2Eu}{\lambda} \right) > 1 \qquad \textnormal{ for every } \lambda < 1.
    \end{equation}
    Indeed, on the one hand if $\|v\|_{\vrho_{\widetilde{V}^{n \times n}_\vphi}} \ge 1$, by Definition \ref{def_modular_space} we have that
    \begin{equation*}
        \inf \left\{ k > 0 : \vrho_{\widetilde{V}^{n \times n}_\vphi}\left( \frac{v}{k} \right) \le 1 \right\} \ge 1,
    \end{equation*}
    entailing that for any $\lambda < 1$ it holds $\|v\|_{\vrho_{\widetilde{V}^{n \times n}_\vphi}} > \lambda$. By \cite[Lemma 3.2.3]{HaH} then this implies that $\vrho_{\widetilde{V}^{n \times n}_\vphi}(\tfrac{v}{\lambda}) > 1$ for any $\lambda < 1$. On the other hand, having $\vrho_{\widetilde{V}^{n \times n}_\vphi}(\tfrac{v}{\lambda}) > 1$ for any $\lambda < 1$ entails that, by Definition \ref{def_modular_space}, 
    \begin{equation*}
        \|v\|_{\vrho_{\widetilde{V}^{n \times n}_\vphi}} =   \inf \left\{ \lambda > 0 \st \vrho_{\widetilde{V}^{n \times n}_\vphi}\left( \frac{v}{\lambda} \right) \le 1 \right\} \ge 1
    \end{equation*}
    directly from the definition of $\inf$. Now, by \cite[Lemma 3.2.3]{HaH} and thanks to the fact that $\vphi^* \in \Phi_c(\Omega)$, see Lemma \ref{lmm:prop_conj}(iii), $\vrho_{\vphi^*}(|w|) \le 1$ if and only if $\|w\|_{\vphi^*} \le 1$, and we can immediately conclude that
    \begin{equation*}
    \begin{aligned}
        \vrho_{\widetilde{V}^{n \times n}_\vphi} \left(\frac{2Eu}{\lambda} \right)
        & \ge \sup \left\{ \int_\Omega \left[ \frac{2 u}{\lambda} \cdot E^* w - \vphi^*(x, |w|) \right] dx \st w \in C^1_c(\Omega; \R^{n \times n}_\sym), \|w\|_{\vphi^*}\le 
        1 \right\} \\
        & \ge \frac{2}{\lambda} \, \widetilde{V}^{n \times n}_\vphi(Eu) - 1 > 1.
    \end{aligned}
    \end{equation*}
    This proves \eqref{eq:mod_magg} and yields the thesis.
\end{proof}

The following result, instead, relates the seminorm $\widetilde{V}^{n\times n}_\vphi$, see Lemma \ref{lmm:seminorm}, with the norm of the symmetrized gradient, in particular considering the norm associated to the associate space $(L^{\vphi^*}(\Omega; \R^{n \times n}))'$. Obviously, in the associate space norm we test with functions in $L^{\vphi^*}(\Omega; \R^{n \times n})$, while in $\widetilde{V}^{n\times n}_\vphi$ the test functions are smooth. This means that some approximation is needed, but a priori we are not able to use density in $L^{\vphi^*}(\Omega; \R^{n \times n})$ since $\vphi^*$ is in general not doubling.
\\
In order to state the result, we need the classical decomposition of the symmetrized gradient, see for instance \cite{BCD},
\begin{equation*}
    Eu = \E u \, \mathcal{L}^n + E^s u.
\end{equation*}
Moreover, let us remark the definition of the space $LD_\loc(\Omega)$, namely
\begin{equation*}
    LD_\loc(\Omega) := \{u \in L^1_\loc(\Omega; \R^n) \st Eu \in L^1_\loc(\Omega;\R^{n \times n}_\sym)\}.
\end{equation*}
The result, analogous to \cite[Theorem 5.2]{EHH}, reads as follows.

\begin{lemma}
\label{lem:BV-gradient}
    Let $\Omega \subset \R^n$ be a bounded, connected and open set. Let $\vphi \in \Phi_w(\Omega)$ and let $u \in LD_\loc(\Omega)$. Then
    \begin{equation*}
        \widetilde{V}^{n \times n}_\vphi(Eu) \le \| \E u \|_{(L^{\vphi^*}(\Omega; \R^{n \times n}))'}.
    \end{equation*}
    Moreover, if additionally $C^1_c(\Omega; \R^{n \times n})$ is dense in $L^{\vphi^*}(\Omega; \R^{n \times n})$, then
    \begin{equation*}
        \widetilde{V}^{n \times n}_\vphi(Eu) = \| \E u \|_{(L^{\vphi^*}(\Omega;\R^{n \times n}))'}.
    \end{equation*}
\end{lemma}
\begin{proof}
    Since $u \in LD_\loc(\Omega)$, it follows from the definition of $\widetilde{V}^{n \times n}_\vphi$ that 
    \begin{equation}
    \label{eq:integrationByParts}
        \widetilde{V}^{n \times n}_\vphi(Eu) = \sup \left\{ \int_\Omega \E u : w \, dx \st w \in C^1_c(\Omega; \R^{n \times n}_\sym), \|w\|_{\vphi^*} \le 1 \right\}. 
    \end{equation}
    The definition of the associate space norm (extended in a standard way to the vector valued setting), see Definition \ref{def:dual}, implies that, chosen $w \in C^1_c(\Omega; \R^{n \times n}_\sym)$ such that $\|w\|_{\vphi^*} \le 1$,
    \begin{equation*}
        \int_\Omega \E u : w \, dx \le \int_\Omega |\E u| |w| \, dx \le \| \E u\|_{(L^{\vphi^*}(\Omega;\R^{n \times n}))'} \|w\|_{\vphi^*}.
    \end{equation*}
    Taking the supremum over those $w$, we conclude that $\widetilde{V}^{n \times n}_\vphi(\E u) \le \|\E u\|_{(L^{\vphi^*}(\Omega;\R^{n \times n}))'}$. 
    \\
    Under the density assumption, namely $C^1_c(\Omega;\R^{n \times n})$ being dense in $L^{\vphi^*}(\Omega;\R^{n \times n})$, we next show the opposite inequality, namely $ \|\E u\|_{(L^{\vphi^*}(\Omega;\R^{n \times n}))'} \le \widetilde{V}^{n \times n}_\vphi(\E u)$. Indeed, let $w \in L^{\vphi^*}(\Omega; \R^{n \times n}_\sym)$ with $\|w\|_{\vphi^*} = 1$ and let $(w_j)_j \subset C^1_c(\Omega; \R^{n \times n}_\sym)$ be a sequence such that $w_j \to w$ in $L^{\vphi^*}(\Omega; \R^{n \times n}_\sym)$ and pointwise almost everywhere. Since we know that also $w_j/\|w_j\|_{\vphi^*} \to w$ in $L^{\vphi^*}(\Omega; \R^{n \times n}_\sym)$, we may assume that $\|w_j\|_{\vphi^*} = 1$ for every $j \in \N$. By Fatou's Lemma, we have
    \begin{equation*}
        \liminf_{j \to \infty} \int_\Omega \E u : w_j \, dx \ge \int_\Omega \E u : w \, dx.
    \end{equation*}
    So, from \eqref{eq:integrationByParts} follows that 
    \begin{equation*}
        \widetilde{V}^{n \times n}_\vphi(\E u) \ge \sup \left\{ \int_\Omega \E u : w \, dx \st w \in L^{\vphi^*}(\Omega; \R^{n \times n}_\sym), \| w\|_{\vphi^*} \le 1 \right\}.
    \end{equation*}
    Let $h\in L^{\vphi^*}(\Omega)$. We can set
    \begin{equation*}
        w :=
        \begin{dcases}
            \frac{\mathcal E u}{|\mathcal E u|} \, h & \textnormal{ if } |\mathcal E u|\neq 0, \\
            0 & \textnormal{ otherwise.}
        \end{dcases}
    \end{equation*}
    This gives 
    \begin{equation*}
        \widetilde{V}^{n \times n}_\vphi( \E u) \ge \sup \left\{ \int_\Omega |\E u| \, h \, dx \st h \in L^{\vphi^*}(\Omega), \| h\|_{\vphi^*}\le 1 \right\} = \| \E u \|_{(L^{\vphi^*}(\Omega;\R^{n \times n}))'}.
    \end{equation*}
    Hence $\widetilde{V}^{n \times n}_\vphi(\E u) = \| \E u \|_{(L^{\vphi^*}(\Omega;\R^{n \times n}))'}$. This concludes the proof.
\end{proof}

In the second part of Lemma \ref{lem:BV-gradient} we assume that $C^1_c(\Omega; \R^{n \times n})$ is dense in $L^{\vphi^*}(\Omega; \R^{n \times n})$. If $\vphi^*$ satisfies (A0) and (aDec), then this holds automatically by \cite[Theorem 3.7.15]{HaH} (note that this theorem is stated for scalar functions, but the vectorial case can be deduced working components-wise thanks to the equivalence of euclidean norms in spaces of finite dimension). We know already that, by Lemma \ref{lmm:prop_conj}(ii) \& (iv), $\vphi^*$ satisfies these conditions if and only if $\vphi$ satisfies (A0) and (aInc), meaning that it is not clear whether this density is there or not in the subsets where the function $\vphi$ is linear.

%--------------------------------------------

\subsection{Explicit expression of the anisotropic modular}
\label{sec:modular}

In this subsection we derive an explicit formula for the dual modular $\vrho_{\widetilde{V}^{n \times n}_\vphi}$ in terms of the modular $\vrho_\vphi$ of the absolutely continuous part of the symmetric gradient and its singular part with weight given by the recession function, see Definition \ref{def:recession}. In particular, analogously to \cite[Theorem 6.4]{EHH}, the main result of the section reads as follows.

\begin{theorem}
\label{thm:exactFormula}
    Let $\Omega \subset \R^n$ be a bounded, connected and open set. Let $\vphi \in \Phi_c(\Omega) \cap C(\Omega \times [0,+\infty))$ satisfy (A0), (aDec) and restricted (VA1). If $u \in BD(\Omega)$, then 
    \begin{equation*}
        \vrho_{\widetilde{V}^{n \times n}_\vphi}(Eu) = \vrho_\vphi(|\E u|) + \int_\Omega \vphi^\infty \, d|E^su|.
    \end{equation*}
    \qed
\end{theorem}

The following result is an immediate consequence of Theorem \ref{thm:exactFormula}, when the $\Phi$-function $\vphi$ does not depend on $x$, i.e. when it is a classical Orlicz $\Phi$-function.

\begin{corollary}
\label{ascor6.5EHH}
    Let $\vphi \in \Phi_c$ be independent of $x$ and satisfy $(aDec)$. If $u \in BD(\Omega)$, then
    \begin{equation*}
        \vrho_{\widetilde{V}^{n \times n}_\vphi}(Eu) = \vrho_\vphi(|\E u|) + \vphi^\infty |E^su|
    \end{equation*}
    in any Borel subset of $\Omega$ and so
    \begin{itemize}
        \item[(1)] $BD^\vphi(\Omega) = BD(\Omega)$ if $\vphi^\infty < \infty$,
        \item[(2)] $BD^\vphi(\Omega) = LD^\vphi(\Omega)$ if $\vphi^\infty = \infty$.
    \end{itemize}
    \qed
\end{corollary}

We observe that in (2) we cannot say that $BD^\vphi(\Omega) = W^{1,\vphi}(\Omega)$, see e.g. \cite[Definition 3]{BHH18}, since the latter equality would follow from a Korn's inequality. But, as proven in \cite[Theorem 1.1]{BD12}, the latter, for classical convex $\Phi$-functions holds if and only if $\vphi$ satisfies both $\nabla_2$ and $\Delta_2$ conditions. Now, recall first that $\vphi \in \Phi_c$ implies that $\vphi= \vphi^{**}$ and that the $\nabla_2$ condition for $\vphi$ is equivalent, by \cite[Definition 2.4.10]{HHbook}, to $\vphi^*$ satisfying $\Delta_2$. Furthermore, \cite[Corollary 2.4.11]{HHbook} states that $\vphi \in \Phi_w$ satisfies $\nabla_2$ if and only if it satisfies $(aInc)$.  By assumption $\vphi$ satisfies $(aDec)$ but not $(aInc)$, thus $\vphi$ satisfies $\Delta_2$ but not $\nabla_2$. As a particular case, one can consider $\vphi(t) = t \log(1+t)$, see \cite[Remark 1.3]{HHbook}.

In order to prove Theorem \ref{thm:exactFormula} we need the following lemma, proven in \cite[Lemma 6.1]{EHH}, which shows the importance of the recession function $\vphi^\infty$, see Definition \ref{def:recession}.

\begin{lemma}
\label{lem:bound}
    Let $\Omega \subset \R^n$ be a bounded, connected and open set. Let $\vphi \in \Phi_c(\Omega)$ satisfy restricted (VA1). If $\psi \in C(\Omega)$ with $\vrho_{\vphi^*}(\psi) < +\infty$, then $|\psi| \le \vphi^\infty$.
    \qed
\end{lemma}

In \cite[Section 3 \& Example A.1]{HHL} it was shown that log-H\"older continuity is not sufficient when working in $BV^{p(\cdot)}$ and here, similarly, the (A1) condition (corresponding to log-H\"older continuity in the general case of convex $\Phi$-functions) is not sufficient in the previous result, as in the following ones, in view of \cite[Example 4.3]{EHH}. That is why we make use of the restricted (VA1) condition, which corresponds to strong log-H\"older continuity in the case of the variable exponent function, see \cite[Proposition 3.2]{EHH}. 
\\
In the next two results we consider the singular and absolutely continuous parts of the symmetric gradient separately. Then we combine them to handle the whole function in Theorem \ref{thm:exactFormula}. Since the proofs are similar to the ones of \cite[Propositions 6.2 \& 6.3]{EHH}, they can be found in our setting in Appendix \ref{app:proof_mod}.

\begin{proposition}
\label{prop:singularPart}
    Let $\Omega \subset \R^n$ be a bounded, connected and open set. Let $\vphi \in \Phi_c(\Omega)$ satisfy (A0), (aDec) and restricted (VA1). If $u \in BD(\Omega)$, then
    \begin{equation*}
        \sup_{\psi \in T^\vphi} \int_\Omega \psi : dE^su = \int_\Omega \vphi^\infty \, d|E^su|,
    \end{equation*}
    where 
    \begin{equation}
    \label{def_test}
        T^\vphi := \left\{ \psi \in C^1_c(\Omega; \R^{n \times n}_\sym) \st \vrho_{\vphi^*}(|\psi|) < +\infty \right\}.
    \end{equation}
    \qed
\end{proposition}

%\color{magenta} In realtà nell'Appendice vorremmo provare la Proposition \ref{prop:singularPart} in ipotesi più deboli se $\varphi(\cdot, t) = t^p(\cdot)$. Quest'ultima proposizione in realt\'a dice di più, dice anche dove $|E^s u|$ \'e concentrata se $u \in BD^\varphi$. Per questa temo ci voglia la continuit\'a

\begin{proposition}
\label{prop:exactFormulaAC}
    Let $\Omega \subset \R^n$ be a bounded, connected and open set. Let $\vphi \in \Phi_c(\Omega) \cap C(\Omega \times [0,+\infty))$ satisfy (A0) and (aDec).
    If $u \in BD(\Omega)$, then 
    \begin{equation*}
        \sup_{\psi \in T^\vphi} \int_\Omega \left[\E u : \psi - \vphi^*(x,|\psi|) \right] dx = \vrho_\vphi(|\E u|).
    \end{equation*}
    \qed
\end{proposition}

Note that removing the assumption that $\vphi$ is continuous in both variables is not trivial.
\\
Having the two previous propositions at hands, we can now prove Theorem \ref{thm:exactFormula}, exploiting the ideas used in \cite[Theorem 6.4]{EHH}.

\begin{proof}[Proof of Theorem \ref{thm:exactFormula}]
    Since $Eu = \E u + E^su$ and since $u \in BD(\Omega)$, an integration by parts implies that 
    \begin{equation*}
        \int_\Omega u \cdot E^*\psi \, dx = \int_\Omega \psi : dEu = \int_\Omega \E u : \psi \, dx + \int_\Omega \psi : dE^su
    \end{equation*}
    for any $\psi \in T^\vphi$. Hence, the claim follows from Propositions \ref{prop:singularPart} and \ref{prop:exactFormulaAC} once we prove that
    \begin{equation*}
    \begin{aligned}
        &\sup_{\psi \in T^\vphi} \left\{\int_\Omega [\E u : \psi - \vphi^*(x, |\psi|)] \, dx + \int_\Omega \psi : dE^su \right\} \\
        = &\sup_{\psi \in T^\vphi} \int_\Omega [\E u : \psi - \vphi^*(x, |\psi|)] \, dx + \sup_{\psi \in T^\vphi} \int_\Omega \psi : dE^su. 
    \end{aligned}
    \end{equation*}
    The first inequality, namely ``$\le$'', is obvious, so we only need to prove the opposite one.
    \\
    Let us first assume that $\vrho_\vphi(|\E u|) + \int_\Omega \vphi^\infty \, d|E^su| < +\infty$ and let us fix $\e > 0$. We can always choose $\psi_1, \psi_2\in T^\vphi$ such that 
    \begin{equation}
    \label{eq_aux:choice_1}
        \sup_{\psi \in T^\vphi} \int_\Omega [\E u : \psi - \vphi^*(x, |\psi|)] \, dx \le \int_\Omega [\E u : \psi_1 - \vphi^*(x, |\psi_1|)] \, dx + \e < +\infty
    \end{equation}
    and
    \begin{equation}
    \label{eq_aux:choice_2}
        \sup_{\psi \in T^\vphi} \int_\Omega \psi : dE^su \le \int_\Omega \psi_2 : dE^su + \e < +\infty.
    \end{equation}
    Since $u \in BD(\Omega)$ and $\psi_1, \psi_2 \in T^\vphi$, we have $|\E u|\,|\psi_i| \in L^1(\Omega)$ and $\vrho_{\vphi^*}(|\psi_i|) < +\infty$ for $i \in \{1,2\}$. Thus, by the absolute continuity of the integral, we can find $\delta > 0$ such that 
    \begin{equation}
    \label{eq_aux:stima_1_2}
        \left| \int_{\Omega \setminus \Omega_1} [\E u :  \psi_i - \vphi^*(x, |\psi_i|)] \, dx \right| \le \int_{\Omega \setminus \Omega_1} [|\E u| \, |\psi_i| + \vphi^*(x, |\psi_i|)] \, dx \le \e
    \end{equation}
    for $i \in \{1,2\}$ and any $\Omega_1 \subset \Omega$ with $|\Omega \setminus \Omega_1|< \delta$ and 
    \begin{equation*}
        \left|\int_{\Omega \setminus \Omega_2} \psi_2 : dE^su \, \right| \le \int_{\Omega \setminus \Omega_2} \vphi^\infty \, d|E^su| \le \e
    \end{equation*}
    for any $\Omega_2 \subset \Omega$ with $|E^su|(\Omega \setminus \Omega_2) < \delta$.
    \\
    Since $\supp \, E^su$ is a set with zero Lebesgue measure, we can find a finite collection of open rectangles $Q_i \subset \Omega$ such that $|E^su|(\bigcup_{i \in \N} Q_i) > |E^su|(\Omega) - \delta$ and $|\bigcup_{i \in \N} 2Q_i| < \delta$. Let us now choose $\theta \in C^1_c(\Omega)$ such that $0 \le \theta \le 1$, $\theta=1$ in $\Omega_2 := \bigcup_{i \in \N} Q_i$ and $\theta = 0$ in $\Omega_1 := \Omega \setminus \bigcup_{i \in \N} 2Q_i$, and let us fix $\psi_\e := \theta \, \psi_2 + (1 - \theta) \psi_1 \in C^1_c(\Omega; \R^{n\times n}_\sym)$. Since $\psi_\e$ is a pointwise convex combination, we get
    \begin{equation*}
        \vphi^*(\cdot, |\psi_\e|) \le \vphi^*(\cdot, \max\{|\psi_2|, |\psi_1|\}) \le \vphi^*(\cdot, |\psi_2|) + \vphi^*(\cdot, |\psi_1|).
    \end{equation*}
    This yields that $\vrho_{\vphi^*}(|\psi_\e|) \le \vrho_{\vphi^*}(|\psi_2|) + \vrho_{\vphi^*}(|\psi_1|) < +\infty$, yielding $\psi_\e \in T^\vphi$. We then get that
    \begin{equation*}
    \begin{aligned}
        \vrho_{\widetilde{V}^{n \times n}_\vphi}(Eu)
        &= \sup_{\psi \in T^\vphi} \left\{\int_\Omega [\E u : \psi - \vphi^*(x, |\psi|)] \, dx + \int_\Omega \psi : dE^su \right\} \\
        &\ge \int_\Omega [\E u : \psi_\e - \vphi^*(x, |\psi_\e|)] \, dx + \int_\Omega \psi_\e : dE^su \\
        &\ge \int_{\Omega_1} [\E u : \psi_1 - \vphi^*(x, |\psi_1|)] \, dx + \int_{\Omega_2} \psi_2 : dE^su - c_\theta \\
        &\ge \int_\Omega [\E u : \psi_1 - \vphi^*(x, |\psi_1|)] \, dx + \int_\Omega \psi_2 : dE^su - 5\e,
    \end{aligned}
    \end{equation*}
    where, thanks to \eqref{eq_aux:stima_1_2},
    \begin{equation*}
        c_\theta := \int_{\Omega \setminus \Omega_1} [|\E u| \, |\psi_\e| + \vphi^*(x, |\psi_\e|)] \, dx + \int_{\Omega \setminus\Omega_2} |\psi_\e|\,  d|E^su| \le 3\e
    \end{equation*}
    by the absolute integrability assumptions and by Lemma \ref{lem:bound}. By the choices of $\psi_1$ and $\psi_2$, namely \eqref{eq_aux:choice_1} and \eqref{eq_aux:choice_2}, we then get
    \begin{equation*}
        \vrho_{\widetilde{V}^{n \times n}_\vphi}(Eu) \ge \sup_{\psi \in T^\vphi} \int_\Omega [\E u : \psi - \vphi^*(x, |\psi|)] \, dx + \sup_{\psi \in T^\vphi} \int_\Omega \psi : dE^su - 7 \e.
    \end{equation*}
    The lower bound follows by letting $\e \to 0^+$. This concludes the proof in the case $\vrho_\vphi(|\E u|) + \int_\Omega \vphi^\infty \, d|E^su| < +\infty$.
    \\
    If, instead, $\vrho_\vphi(|\E u|) = +\infty$ and $\int_\Omega \vphi^\infty \, d|E^su| < +\infty$, we simply estimate
    \begin{equation*}
    \begin{aligned}
        &\sup_{\psi \in T^\vphi} \left\{ \int_\Omega [\E u : \psi - \vphi^*(x, |\psi|)] \, dx + \int_\Omega \psi : dE^su \right\} \\
        \ge &\sup_{\psi \in T^\vphi} \int_\Omega [\E u : \psi - \vphi^*(x, |\psi|)] \, dx - \sup_{\psi \in T^\vphi} \int_\Omega \psi : dE^su \\
        = & \ \vrho_\vphi(|\E u|) - \int_\Omega \vphi^\infty \, d|E^su| = +\infty. 
    \end{aligned}
    \end{equation*}
    We are only left to consider the case $\int_\Omega \vphi^\infty \, d|E^su| = +\infty$. By the proof of Proposition \ref{prop:singularPart}, see Appendix \ref{app:proof_mod}, there exists $\psi_\e \in C^1_c(\Omega; \R^{n \times n}_\sym)$ such that
    \begin{equation*}
        \int_\Omega \psi_\e : dE^su > \frac{1}{\e}
    \end{equation*}
    and $|\psi_\e| \le \frac{\vphi(\cdot, k)}{k}$ for some $k = k_\e$. For any function $\theta : \Omega \to [0,1]$, we have that 
    \begin{equation*}
        \E u : (\theta \psi_\e) - \vphi^*(\cdot, |\theta \psi_\e|) \ge \E u : (\theta \psi_\e) - \vphi^* \left(\cdot, \frac{\vphi(\cdot, k)}{k} \right) \ge - \left(\frac{\vphi^+(k)}{k} \, |\E u| + \vphi^+(k) \right),
    \end{equation*}
    see Lemma \ref{lmm:prop_conj}(iii). Since the function on the right-hand side is integrable, we can choose $\delta_k > 0$ such that its integral over any measurable $A$ with $|A| < \delta_k$ is at least $-1$. Furthermore, since $\supp \, E^su$ has measure zero, we can choose $\theta \in C^\infty_c(\Omega)$ as before to have support with Lebesgue measure at most $\delta_k$ and such that satisfies
    \begin{equation*}
        \int_\Omega (\theta \psi_\e) : dE^su > \frac{1}{2} \int_\Omega \psi_\e : dE^su > \frac{1}{2\e}. 
    \end{equation*}
    Then,
    \begin{equation*}
    \begin{aligned}
        &\sup_{\psi \in T^\vphi} \left\{\int_\Omega [\E u : \psi - \vphi^*(x, |\psi|)] \, dx + \int_\Omega \psi : dE^su \right\} \\
        \ge & \int_\Omega (\theta \psi_\e) : dE^su + \int_\Omega [\E u : (\theta \psi_\e) - \vphi^*(x, |\theta \psi_\e|)] \, dx \ge \frac{1}{2\e} - 1.
    \end{aligned}
    \end{equation*}
    Letting $\e \to +\infty$, the claim follows. This concludes the proof of the theorem.
\end{proof}

%--------------------------------------------

\subsection{The variable exponent case}
\label{sec:BD_px}

In this section we focus on the special $\Phi$-function
\begin{equation}
\label{nostraphi}
    \vphi(x,t) := \frac{1}{p(x)} \, t^{p(x)},
\end{equation}
with $p:\Omega \to [1,+\infty)$ a measurable function. Such  a $p$ is called a {\it variable exponent}. Let us recall that, choosing $\vphi$ as in \eqref{nostraphi} and considering the set
\begin{equation}
\label{set_superlin}
    Y := \{x \in \Omega \st p(x) = 1\},
\end{equation}
we have that the conjugate $\vphi^*$ is of the form
\begin{equation}
\label{phi*def}
    \varphi^*(x,t) =
    \begin{dcases}
        \tfrac{1}{q(x)} \, t^{q(x)} & \textnormal{ if } x \in \Omega \setminus Y, \\
        \infty \chi_{(1, \infty)}(t) & \textnormal{ if } x \in Y,
    \end{dcases}
\end{equation}
see \cite[Example 4.3]{EHH}, where $q$ is the function defined pointwise in $\Omega \setminus Y$ by $\tfrac{1}{q(x)} + \tfrac{1}{p(x)} = 1$ and with the convention that $\infty \cdot 0 = 0$. From now, given a set $A \subset \Omega$, we write
\begin{equation*}
    p^+_A := \underset{x \in A}{\textnormal{ess-sup} \ } p(x) \qquad \textnormal{ and } \qquad p^-_A := \underset{x \in A}{\textnormal{ess-inf} \ } p(x),
\end{equation*}
respectively $q^\pm_A$ for the conjugate exponent function $q$. In the following, we will assume $p^+_\Omega < +\infty$. We start by recalling the following notion, introduced in \cite{BHHZ}.

\begin{definition}
\label{slHY}
    A variable exponent $p$ is \emph{strongly $\log$-H\"older continuous in $Y$} if
    \begin{equation*}
        p(x)-1 \le \frac{\omega(|x-y|)}{\log(e+ \frac{1}{|x - y|})},
    \end{equation*}
    for some $\omega:[0,\infty)\to [0, \infty)$ with $\lim_{t \to 0}\omega(t)=0$ and every $y \in Y$ and $x\in\Omega$.
\end{definition}

At this point, it is worth observing that the above property is weaker than the 'classical $\log$-H\"older continuity' as stated in \cite[Definition 2.3]{BHHZ}, which is, in turn, equivalent to the restricted (VA1) condition, introduced in Definition \ref{restrVA1def}, for the function \eqref{nostraphi}. 
\\
In \cite{BHHZ} the following technical result has been proven.

\begin{lemma}
\label{lem:smallOnY}
    Suppose that $p$ is strongly $\log$-Hölder continuous in $Y$. If $\rho_{q(\cdot)}(w)<\infty$ for some continuous $w$, then $|w|\le 1$ in $Y$.
\end{lemma}

\begin{remark}
\label{rmk:prop_variable_exp}
    The explicit definition of the function $\vphi$ in \eqref{nostraphi} entails some additional properties other than the explicit form of its conjugate function given by \eqref{phi*def}. Indeed, it is immediate to see that $\vphi \in \Phi_c(\Omega)$, see Definition \ref{def:phi-functions} and \cite[Section 7.1]{HaH}. Moreover, arguing as in \cite[Lemma 7.1.1]{HaH}, it is easy to see that $\vphi$ satisfies (A0), (aInc)$_{p^-_\Omega}$ and (aDec)$_{p^+_\Omega}$, see Definition \ref{prop_Orlicz}.
\end{remark}

The first result we prove in this subsection is the equivalent of Theorem \ref{thm:exactFormula} in this specific setting. In particular, we show that the same conclusion of Theorem \ref{thm:exactFormula} holds under less restrictive hypotheses on our function $\vphi$. Before doing so, we need to recall the definition of the modular associated to the variable exponent space when $p^+_\Omega < +\infty$, given in an equivalent way to the one in \cite[Subsection 2.2]{HHL}, namely
\begin{equation*}
    \vrho_{L^{p(\cdot)}(\Omega; \R^{n \times n}_\sym)}(u) := \int_{\Omega} \frac{1}{p(x)} |u(x)|^{p(x)} \, dx
\end{equation*}
for any $u \in \M(\Omega; \R^{n \times n}_\sym)$. Now, the theorem reads as follows.

\begin{theorem}
\label{thm:eq_mod_px}
    Let $\Omega \subset \R^n$ be a bounded, connected and open set. Let $\vphi$ be defined as in \eqref{nostraphi} and let $p : \Omega \to [1, +\infty)$ be  continuous in $\Omega$ and strongly $\log$-H\"older continuous in $Y$, \color{black} with $p^+_\Omega < +\infty$. %and such that $Y$ is closed, see \eqref{set_superlin}. 
    If $u \in BD(\Omega)$, then 
    \begin{equation}
    \label{rho_old}
        \vrho_{\widetilde{V}^{n \times n}_\vphi}(Eu) = \vrho_{\textnormal{old}}(Eu) := |Eu|(Y) + \vrho_{L^{p(\cdot)}(\Omega \setminus Y; \R^{n \times n}_\sym)}(\E u).
    \end{equation}
\end{theorem}
\begin{proof}
    First, let us observe that the claim of Lemma \ref{lem:bound} is automatically satisfied in the case of $\vphi$ given by \eqref{nostraphi} with the only requirement of strong $\log$-H\"older continuity in $Y$ and \color{black} the assumption $p^+_\Omega$ is finite. Indeed, on the one hand, in view of \eqref{phi*def} and Definition \ref{def:recession} (which is a limit in this specific case), it suffices to prove the inequality in the set $Y$ given by \eqref{set_superlin}, since in $\Omega \setminus Y$ the recession function is $+\infty$. On the other hand though, whenever $\psi \in C(\Omega)$ with $\vrho_{\vphi^*}(|\psi|) < +\infty$, then by \eqref{phi*def} we have that $\vphi^*(x, |\psi|) = 0$ for almost every $x \in Y$, implying that $|\psi| \le 1=\vphi^\infty(x)$ in $Y$, in view of Lemma \ref{lem:smallOnY}, \color{black} which
    proves our claim.
    \\
    As a second step, let us remark that choosing $\vphi$ as in \eqref{nostraphi}, with $p(\cdot)$ continuous, \color{black}, Proposition \ref{prop:exactFormulaAC} is  valid. \color{black} % if we remove the assumptions therein and, again, we only require that $p^+_\Omega < +\infty$.
    In fact, by assumption  $\vphi$ is continuous, and by \color{black} Remark \ref{rmk:prop_variable_exp}, we know that (A0) and (aDec) are satisfied.
   % and this implies that we can apply \cite[Theorem 3.7.15]{HaH}, meaning the required continuity on $\vphi$ can be bypassed exploiting the density of $C^\infty_c(\Omega)$ in $L^\vphi(\Omega)$. 
   Hence, we can conclude that
    \begin{equation*}
        \sup_{\psi \in T^\vphi} \int_\Omega \left[\E u : \psi - \vphi^*(x,|\psi|) \right] \, dx = \frac{1}{p(x)} \int_\Omega |\E u|^{p(x)} \, dx
    \end{equation*}
    for any $u \in BD(\Omega)$, where $T^\vphi$ is defined as in \eqref{def_test}.
    \\
    We are left to prove that Proposition \ref{prop:singularPart} holds true under these hypotheses and choosing $\vphi$ as in \eqref{nostraphi}, namely that
    \begin{equation*}
        \sup_{\psi \in T^\vphi} \int_\Omega \psi : dE^su = \int_\Omega \vphi^\infty \, d|E^su|.
    \end{equation*}
    Note that the inequality "$\le$" is immediately verified thanks to the fact that Lemma \ref{lem:bound} holds under our assumptions. So it suffices to prove that
    \begin{equation*}
        \int_\Omega \varphi^\infty d |E^s u|\leq \sup_{\psi \in T^\varphi}\int_\Omega \psi : d E^s u.
    \end{equation*}
    Thanks to the continuity assumption on $p$ in $\Omega$, \color{black} we obtain that \eqref{leqhk} and \eqref{geqhk} hold. In particular, the second one holds for any $\beta \in [0,1]$, due to the non decreasing monotonicity of $(\cdot)^{p(x)}$. Observe that $h_k$ is lower semicontinuous by definition (it is  a lower semicontinuous envelope, see \cite[Chapter 1]{DM}) and increasing in $k$. Consequently the same arguments as in the proof of Proposition \ref{prop:singularPart} lead to the desired conclusion. Moreover, let us remark that for every $u \in BD(\Omega)$ it holds
    \begin{equation*}
        \sup_{\psi \in T^\vphi} \int_\Omega \psi : d E^s u = \int_Y d|E^s u| = |E^s u|(Y),
    \end{equation*}
    since $\vphi^\infty(x) = 1$ in $Y$ as observed before, which indeed says that $E^s u$ is concentrated on $Y$.
    \\
    Finally, following the same steps of the proof of Theorem \ref{thm:exactFormula}, since under our hypotheses the equivalent formulations of Lemma \ref{lem:bound}, Proposition \ref{prop:singularPart} and Proposition \ref{prop:exactFormulaAC} hold, we get the thesis.
\end{proof}

\begin{remark}
   As a corollary of Theorem \ref{thm:eq_mod_px}, under the same hypotheses we get the coincidence of the spaces $BD^\vphi$ and $BD \cap LD^{p(\cdot)}$, with
   \begin{equation*}
        LD^{p(\cdot)}(U) := \{u \in L^{p(\cdot)}(U; \R^n) \st Eu \in L^{p(\cdot)}(U;\R^{n \times n}_\sym)\},
    \end{equation*}
    when $\vphi$ is as in \eqref{nostraphi}. Alternatively, this latter equality can be proven directly arguing as in \cite[Theorem 3.8]{BHHZ}  even removing the continuity hypotesis on $p(\cdot)$, replacing it by lower semicontinuity \color{black}.
\end{remark}

\section{The imaging problem}
\label{sec:imaging}

In this section we finally deal with the novel regularizer based on a Musielak-Orlicz generalization of classical Total Generalized Variation functionals that we defined in \eqref{TGV}, namely the anisotropic Total Generalized Variation of order $2$, with weight $\alpha = (\alpha_1, \alpha_2) \in \R^+ \times \R^+$, given by
\begin{equation*}
\begin{aligned}
    TGV^{\vphi,2}_\alpha(u):= \sup\bigg\{ &\int_\Omega u \,\div^2 \psi \, dx \st \\
    & \ \psi \in C^2_c (\Omega; \R^{n \times n}_\sym), \|\psi\|_{\vphi^*} \le \alpha_1, \|\div \, \psi\|_{\vphi^*} \le \alpha_2 \bigg\},
\end{aligned}
\end{equation*}
where we consider $u \in L^1_\loc(\Omega)$.
\\
In Subsection \ref{sec:tgv_prop} we prove some properties of this functional, while in Subsection \ref{sec:dual_formulations} we give a dual characterization of this functional exploiting the space $BD^\vphi$ we defined in Section \ref{sec:funct_anal}. Finally, in Subsection \ref{sec:ex_stab} we show how the anisotropic Total Generalized Variation inherits classical existence and stability properties.

%--------------------------------------------

\subsection{Basic properties of \texorpdfstring{$TGV^{\vphi,2}_\alpha(u)$}{TGV}}
\label{sec:tgv_prop}

Following the lines of \cite{BKP10}, we define the space of functions of  bounded variation of order $2$ with Orlicz growth and weight $\alpha$ as
\begin{equation}
\label{def:BGV}
    BGV^{\vphi,2}_\alpha(\Omega) := \{ u \in L^\vphi(\Omega) \st TGV^{\vphi,2}_\alpha(u) < +\infty \},
\end{equation}
endowed with the norm
\begin{equation*}
    \|u\|_{BGV^{\vphi,2}_\alpha} := \|u\|_\vphi + TGV^{\vphi,2}_\alpha(u),
\end{equation*}
where the functional $TGV^{\vphi,2}_\alpha$ is defined as in \eqref{TGV}.

\begin{remark}
    Note that it would also be possible to define the functionals $TGV^{\vphi,k}_\alpha$ and the spaces $BGV^{\vphi,k}_\alpha(\Omega)$ for general $k \in \N^+$ following the lines of \cite{BKP10}. Moreover, for $k = 1$ and $\alpha \in \R^+$, and chosen $u \in L^1_\loc(\Omega)$, we would have
    \begin{equation*}
        TGV^{\vphi,1}_\alpha(u) = \sup \left\{\int_\Omega u \, \div \psi \, dx \st \psi \in C^1_c (\Omega; \R^n), \|\psi\|_{\vphi^*} \le \alpha \right\} = \alpha V_\vphi(u),
    \end{equation*}
    where the right-hand side is defined as in \eqref{def:EHH_TV}. Hence, we can indeed speak of a generalization of an anisotropic total variation.  Contrary to Remark \ref{rk:v-equiv}, which was true for $u \in BV(\Omega)$ in order to know at least that $Du \in \M(\Omega;\R^n)$, here it is enough to take $u \in L^1_{\loc}(\Omega)$.
\end{remark}

In the following result we prove some basic properties of the anisotropic Total Generalized Variation.

\begin{proposition}
\label{prop:TGV_prop}
    Let $\Omega$ be a bounded, connected and open set in $\R^n$, with $n \ge 2$. Let $\vphi \in \Phi_w(\Omega)$ satisfy (A0) and let $\alpha = (\alpha_1, \alpha_2) \in \R^+ \times \R^+$. Then, the following statements hold:
    \begin{itemize}
        \item[(i)] $TGV^{\vphi,2}_\alpha$ is a seminorm on the normed space $BGV^{\vphi,2}_\alpha(\Omega)$.
        
        \item[(ii)] For $u \in L^1_\loc(\Omega)$, there holds $TGV^{\vphi,2}_\alpha(u) = 0$ if and only if $u$ is a polynomial of degree $0$ or $1$.
        
        \item[(iii)] For eevery $\tilde\alpha = (\tilde\alpha_1, \tilde\alpha_2) \in \R^+ \times \R^+$, the seminorms $TGV^{\vphi,2}_\alpha$ and $TGV^{\vphi,2}_{\tilde\alpha}$ are equivalent.
        
        \item[(iv)] $TGV^{\vphi,2}_\alpha$ is rotationally invariant, namely for any matrix $R \in SO(n)$ and $u \in BGV^{\vphi,2}_\alpha(\Omega)$ we can define $\tilde{u}(x) := u(Rx)$ and we have that $\tilde{u} \in BGV^{\vphi,2}_\alpha(R^\T\Omega)$ and $TGV^{\vphi,2}_\alpha(\tilde{u}) = TGV^{\vphi,2}_\alpha(u)$.
        
        \item[(v)] For $r > 0$ and $u \in BGV^{\vphi,2}_\alpha(\Omega)$, define $\bar{u}(x) := u(rx)$. Then $\bar{u} \in BGV^{\vphi,2}_\alpha(r^{-1}\Omega)$ and
        \begin{equation*}
            TGV^{\vphi,2}_\alpha(\bar{u}) = r^{-n} \, TGV^{\vphi,2}_{\bar\alpha}(u) \qquad \textnormal{ with } \bar\alpha = (\bar\alpha_1, \bar\alpha_2) := (\alpha_1 r^2, \alpha_2 r).
        \end{equation*}
    \end{itemize}
\end{proposition}
\begin{proof}
    {\it (i).} Thanks to the definition of the functional $TGV^{\vphi,2}_\alpha$ given in \eqref{TGV}, it is obvious to see that the two conditions (N2) and (N3) in Definition \ref{def:norms} hold, implying that $TGV^{\vphi,2}_\alpha$ is a seminorm. Note that $BGV^{\vphi,2}_\alpha(\Omega)$ is a normed space because $\vphi \in \Phi_w(\Omega)$ and satisfies (A0), so $\|\cdot\|_\vphi$ is a norm, see \cite[Theorem 3.7.13]{HaH}.
    \\
    {\it (ii).} Let us denote the set of the test functions for the functional $TGV^{\vphi,2}_\alpha$ by
    \begin{equation*}
        K^{\vphi,2}_\alpha(\Omega) := \{ \psi \in C^2_c (\Omega; \R^{n \times n}_\sym) \st \|\psi\|_{\vphi^*} \le \alpha_1, \|\div \, \psi\|_{\vphi^*} \le \alpha_2 \}.
    \end{equation*}
    Let us first assume that $u$ is a polynomial of degree $0$ or $1$, implying that $\nabla^2 u = 0$. Then, since the test function $\psi$ is in $K^{\vphi,2}_\alpha(\Omega)$ and thanks to two integration by parts, it is obvious that $TGV^{\vphi,2}_\alpha(u) = 0$.
    \\
    Suppose now that $TGV^{\vphi,2}_\alpha(u) = 0$ for some $u \in L^1_\loc(\Omega)$. For any $\psi \in C^2_c (\Omega; \R^{n \times n}_\sym)$, one can find a $\lambda > 0$ such that
    \begin{equation*}
        \|\lambda \psi\|_\infty \le \alpha_1 \quad \textnormal{ and } \quad \|\div (\lambda \psi)\|_\infty \le \alpha_2.
    \end{equation*}
    Since $\vphi$ satisfies (A0), meaning that $\vphi^*$ satisfies (A0) as well, see Lemma \ref{lmm:prop_conj}(iv), and since $\Omega$ is bounded, thanks to Lemma \ref{lmm:inclusions} we have that $\lambda\psi \in K^{\vphi,2}_\alpha(\Omega)$. Testing in \eqref{TGV} with $\lambda\psi$, we then get
    \begin{equation*}
        \int_\Omega u \,\div^2 \psi \, dx = 0 \qquad \textnormal{ for any } \psi \in C^2_c (\Omega; \R^{n \times n}_\sym),
    \end{equation*}
    implying that $\nabla^2 u = 0$ in the weak sense. Since $\Omega$ is connected, this implies that $u$ has to be a polynomial of degree less than $2$.
    \\
    {\it (iii).} The equivalence of the seminorms $TGV^{\vphi,2}_\alpha$ and $TGV^{\vphi,2}_{\tilde\alpha}$ can be proven noticing that, chosen
    \begin{equation*}
        c := \frac{\max\{\alpha_1, \alpha_2\}}{\min\{\tilde\alpha_1, \tilde\alpha_2\}},
    \end{equation*}
    we have that $c \, K^{\vphi,2}_\alpha(\Omega) \subset K^{\vphi,2}_{\tilde\alpha}(\Omega)$, meaning that $c \, TGV^{\vphi,2}_\alpha(u) \le TGV^{\vphi,2}_{\tilde\alpha}(u)$ for any $u \in L^1_\loc(\Omega)$. Interchanging the roles of $\alpha$ and $\tilde\alpha$ gives us the desired equivalence.
    \\
    {\it (iv).} Thanks to \cite[(A.4)]{BKP10}, we know that given a matrix $R \in SO(n)$, then
    \begin{equation*}
        \psi \in C^2_c (R^\T\Omega; \R^{n \times n}_\sym) \quad \iff \quad \widetilde\psi := (\psi \circ R^\T)R^\T \in C^2_c (\Omega; \R^{n \times n}_\sym),
    \end{equation*}
    with $\div^l \, \widetilde\psi = ((\div^l \, \psi) \circ R^\T)R^\T$, with $l = 0,1$ where $\div^0 = id$. Hence, for $l = 0,1$,
    \begin{equation*}
        \|\div^l \, \widetilde\psi\|_{\vphi^*} = \| ((\div^l \, \psi) \circ R^\T)R^\T \|_{\vphi^*} = \| (\div^l \, \psi) \circ R^\T \|_{\vphi^*} = \| \div^l \, \psi \|_{\vphi^*},
    \end{equation*}
    implying that $\psi \in K^{\vphi,2}_\alpha(R^\T\Omega)$ if and only if $\widetilde\psi \in K^{\vphi,2}_\alpha(\Omega)$. Lastly, for any $\psi \in K^{\vphi,2}_\alpha(R^\T\Omega)$ we have
    \begin{equation*}
        \int_{R^\T\Omega} u(Rx) \,\div^2 \psi(x) \, dx = \int_\Omega u(x) \,\div^2 \psi(R^\T x) \, dx = \int_\Omega u(x) \, \div^2 \widetilde\psi(x) \, dx,
    \end{equation*}
    yielding $TGV^{\vphi,2}_\alpha(u \circ R) = TGV^{\vphi,2}_\alpha(u)$.
    \\
    {\it (v).} Note that
    \begin{equation*}
        \div (v \circ r^{-1} \I) = r^{-1} (\div \, v) \circ r^{-1} \I,
    \end{equation*}
    implying that, chosen $\psi \in C^l_c(r^{-1}\Omega;\R^{2 \times 2}_\sym)$ and $\bar\psi := r^2 \, \psi \circ r^{-1} \I$, we have
    \begin{equation*}
        \| \div^l \bar\psi \|_{\vphi^*} = r^{2-l} \| \div^l \psi \|_{\vphi^*},
    \end{equation*}
    for $l = 0,1$. As a consequence, $\psi \in K^{\vphi,2}_\alpha(r^{-1} \Omega)$ if and only if $\bar\psi \in K^{\vphi,2}_{\bar\alpha}(\Omega)$ with $\bar\alpha = (\bar\alpha_1, \bar\alpha_2) := (\alpha_1 r^2, \alpha_2 r)$. Lastly, this yields
    \begin{equation*}
        \int_{r^{-1} \Omega} u(rx) \,\div^2 \psi(x) \, dx = r^{-n} \int_\Omega u(x) \,\div^2 \psi(r^{-1} x) \, dx = \int_\Omega u(x) \, \div^2 \bar\psi(x) \, dx,
    \end{equation*}
    implying $TGV^{\vphi,2}_\alpha(u \circ r \, \I) = r^{-n} \, TGV^{\vphi,2}_{\bar\alpha}(u)$.
\end{proof}

Note that it is possible to prove that $TGV^{\vphi,2}_\alpha$ is weakly lower semicontinuous with respect to the weak $L^p$ convergence, as well as to the weak* convergence in $L^\infty$.

\begin{lemma}
\label{lmm:lsc_TGV}
    Let $\Omega$ be a bounded set in $\R^n$ and let $\vphi \in \Phi_w(\Omega)$. Then, 
    $TGV^{\vphi,2}_\alpha$ is weakly lower semicontinuous in $L^p(\Omega)$ for every $1 \le p < +\infty$, as well as weak* lower semicontinuous in $L^\infty(\Omega)$ for any $\alpha \in \R^+ \times \R^+$.
\end{lemma}
\begin{proof}
    We only prove the lemma for $p \in [1, +\infty)$, as the case $p = +\infty$ is analogous.
    \\
    Chosen $\psi \in C^2_c(\Omega; \R^{n \times n}_\sym)$ with $\|\psi\|_{\vphi^*} \le \alpha_1$ and $\|\div \, \psi\|_{\vphi^*} \le \alpha_2$ as in \eqref{TGV}, then $\div^2 \psi \in L^\infty(\Omega)$. Then, chosen $(u_k)_k \subset L^p(\Omega)$ such that $u_k \rightharpoonup u$ weakly in $L^p(\Omega)$ we get
    \begin{equation*}
        \int_\Omega u \, \div^2\psi \, dx = \liminf_{k \to +\infty} \int_\Omega u_k \, \div^2\psi \, dx \le \liminf_{k \to +\infty} TGV^{\vphi,2}_\alpha(u_k).
    \end{equation*}
    The thesis follows by taking the supremum over all such $\psi$.
\end{proof}

% The above result holds true also in the case the weak convergence in $L^p(\Omega)$ is replaced by the weak convergence in $L^\phi$ for any Sobolev-Orlicz function $\phi$.

The next proposition is the equivalent of \cite[Proposition 3.5]{BKP10} in our setting, proving that $BGV^{\vphi,2}_\alpha(\Omega)$ is a Banach space.

\begin{proposition}
    Let $\Omega$ be a bounded, connected and open set in $\R^n$, with $n \ge 2$. Let $\vphi \in \Phi_w(\Omega)$ satisfy (A0). Then, for any weight $\alpha = (\alpha_1, \alpha_2) \in \R^+ \times \R^+$ the space $BGV^{\vphi,2}_\alpha(\Omega)$ is a Banach space if endowed with the norm $\|\cdot\|_{BGV^{\vphi,2}_\alpha}$.
\end{proposition}
\begin{proof}
    Thanks to Proposition \ref{prop:TGV_prop}(iii), we know that all the seminorms $TGV^{\vphi,2}_\alpha$, respectively norms $\|\cdot\|_{BGV^{\vphi,2}_\alpha}$, are equivalent for any $\alpha \in \R^+ \times \R^+$, so we only need to prove that $BGV^{\vphi,2}_\alpha(\Omega)$ is complete for some $\alpha \in \R^+ \times \R^+$.
    \\
    Note that, since $\Omega$ is bounded and $\vphi$ satisfies (A0), then by Lemma \ref{lmm:inclusions} we have $L^\vphi(\Omega) \hookrightarrow L^1(\Omega)$. This, together with Lemma \ref{lmm:lsc_TGV}, implies that $TGV^{\vphi,2}_\alpha$ is a lower semicontinuous functional with respect to $L^\vphi(\Omega)$ for any $\alpha$ and $\vphi$ chosen as in the hypotheses.
    \\
    Now let $(u_n)_n$ be a Cauchy sequence in $BGV^{\vphi,2}_\alpha(\Omega)$. This yields that $(u_n)_n$ is a Cauchy sequence also in $L^\vphi(\Omega)$, implying the existence of a limit $u \in L^\vphi(\Omega)$ such that
    \begin{equation*}
        TGV^{\vphi,2}_\alpha(u) \le \liminf_{n \to +\infty} TGV^{\vphi,2}_\alpha(u_n),
    \end{equation*}
    thanks to the lower semicontinuity of the functional. As a consequence, $u \in BGV^{\vphi,2}_\alpha(\Omega)$ and the only thing left to be shown is that $u$ is the limit in the corresponding norm as well. In order to do so, let us choose $\e > 0$ and $n \in \N$ such that for any $\bar{n} \ge n$ it holds that $TGV^{\vphi,2}_\alpha(u_n - u_{\bar{n}}) \le \e$. Thanks again to the lower semicontinuity of the functional (with respect to the $L^1$ weak convergence), letting $\bar{n} \to +\infty$ we then get
    \begin{equation*}
        TGV^{\vphi,2}_\alpha(u_n - u) \le \liminf_{\bar{n} \to +\infty} TGV^{\vphi,2}_\alpha(u_n - u_{\bar{n}}) \le \e,
    \end{equation*}
    implying that $u_n \to u$ in $BGV^{\vphi,2}_\alpha(\Omega)$.
\end{proof}

%--------------------------------------------

\subsection{Dual formulation}
\label{sec:dual_formulations}

The first step in order to treat minimization problems in image denoising with this new regularizer is to prove an equivalent result to \cite[Theorem 3.1]{BV20}, namely proving that \eqref{TGV} can be rewritten equivalently as a minimization problem which can be interpreted as an optimal balancing between the first and second derivative of $u$ in terms of “sparse” penalization, see \cite[Remark 3.2]{BV20}. In \cite{BV20} this was done via the Radon norm (i.e. the norm in $\mathcal M$ associated to the Radon measure $Du$), while here we need to use the variant of the generalized variation associated to $\vphi$ defined in \eqref{V_tilde}.  Given the anisotropic nature of such regularizers, we expect them to better capture fine texture details in the images. The result, providing a dual characterization of $TGV^{\vphi,2}_\alpha$, reads as follows.

\begin{theorem}
\label{thm:equivalence_TGV}
    Let $\Omega \subset \R^n$ be a bounded, connected, $C^1$ and open set. Let $\vphi \in \Phi_w(\Omega)$ satisfy (A0) and let $\alpha = (\alpha_1, \alpha_2) \in \R^+ \times \R^+$. Then for any $u \in L^1(\Omega)$ we have
    \begin{equation}
    \label{eq:dualTGV}
        TGV^{\vphi, 2}_\alpha(u) = \min_{w \in BD(\Omega)} \left\{ \alpha_2 \widetilde{V}^n_\vphi (Du - w) + \alpha_1 \widetilde{V}^{n \times n}_\vphi(E w)\right\},
    \end{equation}
    where $TGV^{\vphi, 2}_\alpha$ is defined as in \eqref{TGV} and where we have adopted the convention that the above right-hand side is infinite if $\widetilde{V}^{n \times n}_\vphi(E w) = +\infty$. If, additionally, $u \in BV^\vphi(\Omega)$, defined as in \eqref{def:BV_phi}, then
    \begin{equation}
    \label{eq:dualTGV_phi}
        TGV^{\vphi, 2}_\alpha(u) = \min_{w \in BD^\vphi (\Omega)} \left\{ \alpha_2 \widetilde{V}^n_\vphi (Du - w) + \alpha_1 \widetilde{V}^{n \times n}_\vphi(E w)\right\},
    \end{equation}
    where the space $BD^\vphi(\Omega)$ is defined as in \eqref{BD_phi}.
\end{theorem}
\begin{proof}
    Let us define $X_2 := C^2_0(\Omega; \R^{n \times n}_\sym)$, $X_1 := C^1_0(\Omega; \R^n)$ and $\Lambda := \div \in \mathcal{L}(X_2,X_1)$. We stress that
    \begin{equation}
    \label{chiusure}
        X_2 = \overline{C^2_c(\Omega;\R^{n \times n}_\sym)}^{\|\cdot\|_{C^2}}
        \qquad \textnormal{ and } \qquad
        X_1 = \overline{C^1_c(\Omega;\R^n) }^{\|\cdot\|_{C^1}}.
    \end{equation}
    Chosen $u \in L^1(\Omega)$, we define
    \begin{equation*}
        F_1(M) := I_{\{ \|\cdot\|_{\vphi*} \le \alpha_2\}}(M), \quad \textnormal{ for any } M \in X_2
    \end{equation*}
    and
    \begin{equation*}
        F_2^u(v) := I_{\{\| \cdot\|_{\vphi^*} \le \alpha_1\}}(v) - \int_\Omega u\,\div v \, dx, \quad \textnormal{ for any } v \in X_1.
    \end{equation*}
    Since $\vphi$ satisfies (A0), then by Lemma \ref{lmm:prop_conj}(iv) also $\vphi^*$ satisfies (A0) and, by Lemma \ref{lmm:inclusions}, we have that $F_1$ and $F^u_2$ are continuous in the $C^2$- and $C^1$-topologies, respectively. Then, thanks to a density argument and to \eqref{chiusure}, we have that
    \begin{equation*}
        TGV^{\vphi, 2}_\alpha(u) = -\inf_{M \in X_2} \left[ F_1(M) + F_2^u(\Lambda M) \right].
    \end{equation*}
    Let us remark that we can write the space $X_1$ equivalently as
    \begin{equation}
    \label{unione_domini}
        X_1 = \bigcup_{\lambda \ge 0} \lambda \left(\dom F_2^u - \Lambda \, \dom F_1\right).
    \end{equation}
    Indeed, if $z \in \dom F_2^u$ then, trivially, $z \in X_1$ implies $z \in \bigcup_{\lambda \ge 0}\lambda\left(\dom F_2^u - \Lambda \, \dom F_1\right)$ and for some $\lambda \ge 0$ it is always verified that $\tfrac{z}{\lambda} \in \dom F^u_2$. This implies that $X_1 \subset \bigcup_{\lambda \ge 0}\lambda\left(\dom F_2^u - \Lambda \, \dom F_1\right)$. On the other hand, the converse is true since $X_1$ is a closed vector space (and $\dom F^u_2$ and $\dom F_1$ are subsets of $X_1$).
    \\
    Now, since $X_1$ is a closed set, it means that also the right-hand side of \eqref{unione_domini} is closed, which allows us to use \cite[Corollary 2.3]{AtBr86}, a Fenchel-Rockafellar duality result, getting
    \begin{equation*}
        TGV^{\vphi, 2}_\alpha(u)
        = -\inf_{M \in X_2} \left[F_1(M) + F_2^u(\Lambda M) \right]
        = \min_{w \in X_1^*} \left[ (F_2^u)^*(w) + F_1^*(-\Lambda^* w) \right].
    \end{equation*}
    Now, exploiting the definition of Fenchel conjugate, we get
    \begin{equation*}
        (F_2^u)^*(w) = \alpha_2 \sup_{v \in X_1} \left\{ \langle w , v \rangle_Y + \int_\Omega u \, \div v \, dx \st \|v\|_{\vphi^*} \le 1 \right\},
    \end{equation*}
    for every $w\in {\mathcal D}'(\Omega;\R^n)$, and
    \begin{equation*}
        F_1^*(-\Lambda^* w) = \alpha_1 \sup_{M \in X_2} \left\{ \langle -\Lambda^* w, M \rangle_X \st \|M\|_{\vphi^*} \le 1 \right\},
    \end{equation*}
    for every $w\in {\mathcal D}'(\Omega;\R^n)$ with $-\Lambda^* w \in {\mathcal D}'(\Omega;\R^{n \times n})$, where $\langle\cdot,\cdot\rangle_{X_1}$ and $\langle\cdot,\cdot\rangle_{X_2}$ denote, respectively, the duality products between $X_1$ and $X_1^*$ and between $X_2$ and $X_2^*$. Now, we observe that if $F_1^*(-\Lambda^* w) = +\infty$, then by Lemma \ref{lmm:inclusions} we find $\|-\Lambda^* w\|_\M = +\infty$. Moreover, since $-\Lambda^* = E$, we infer that for every $w\in {\mathcal D}'(\Omega;\R^n)$ with $-\Lambda^* w \in {\mathcal D}'(\Omega;\R^{n \times n})$ we have
    \begin{equation*}
        F_1^*(-\Lambda^* w) = 
        \begin{dcases} 
            \alpha_1 \widetilde{V}_\vphi^{n \times n}(E w) & \textnormal{if } E w\in \M (\Omega;\R^{n \times n}) \\
            +\infty & \textnormal{if } E w \in {\mathcal D}'(\Omega;\R^{n\times n}) \setminus \M(\Omega;\R^{n \times n}).
        \end{dcases}
    \end{equation*}
    Analogously, by observing that, since $u\in L^1(\Omega)$, 
    \begin{equation*}
        (F_2^u)^*(w) = \alpha_2 \sup_{v \in X_1} \left\{ \langle Du-w , v\rangle_{X_1} \st \|v\|_{\vphi^*} \le 1 \right\},
    \end{equation*}
    for every $w\in {\mathcal D}'(\Omega;\R^n)$, the same argument as for $F_1^*$ yields
    \begin{equation*}
        (F_2^u)^*(w) = 
        \begin{dcases} 
            \alpha_2 \widetilde{V}_\vphi^n (Du-w) &\textnormal{if } Du - w \in \M(\Omega;\R^n) \\
            +\infty & \textnormal{if } Du-w \in {\mathcal D}'(\Omega;\R^n) \setminus \M(\Omega;\R^n),
        \end{dcases}
    \end{equation*}
    again for every $w\in {\mathcal D}'(\Omega;\R^n)$. Thus, to summarize, by classical properties of $BD$ and $BV$, see \cite[Chapter II, Theorem 2.3]{temam} and \cite{AFP}, we have that the following characterization holds:
    \begin{equation}
    \label{eq_min_inf}
        TGV^{\vphi, 2}_\alpha(u) = 
        \begin{dcases}
            \min_{w \in BD(\Omega)} \left\{ \alpha_2 \widetilde{V}_\vphi^n (Du-w) + \alpha_1 \widetilde{V}_\vphi^{n \times n}(E w) \right\} & \textnormal{if } u\in BV(\Omega;\R^n), \\
            +\infty & \textnormal{otherwise in } L^1(\Omega;\R^n).
        \end{dcases}
    \end{equation}
    To conclude the proof we observe that if $u \in BV^\vphi(\Omega)$, then by \cite[Theorem 3.4.6]{HaH} the above infimum problem can be considered on the subspace of $BD(\Omega)$ where additionally $\|w\|_{\vphi^{**}} < +\infty$, namely $BD^\vphi(\Omega)$, see Remark \ref{rmk:BD-phi=BD-tilde}.
\end{proof}

Note that in the proof we exploited the condition (A0) on the function $\vphi$ in order to have the inclusions $BV^\vphi(\Omega) \hookrightarrow BV(\Omega)$ and $BD^\vphi(\Omega) \hookrightarrow BD(\Omega)$, see Remarks \ref{rmk:A0_si_no} and \ref{rmk:A0_si_no_BD}.

\noindent The following result reduces the functional-analytic setting to the space $BV^\vphi(\Omega)$, see \eqref{def:BV_phi}, by establishing equivalence of the two spaces $BGV^{\vphi,2}_\alpha$ and $BV^\vphi$.

\begin{corollary}
\label{cor:quasi_equiv}
    Under the same assumptions of Theorem \ref{thm:equivalence_TGV}, we have that $TGV^{\vphi,2}_\alpha(u) < +\infty$ if and only if $u\in BV^\vphi(\Omega)$.
\end{corollary}
\begin{proof}
    By Remark \ref{rk:v-equiv}, choosing $w = 0$ as a test function in the characterization of $TGV^{\vphi,2}_\alpha(u)$, namely \eqref{eq:dualTGV}, we directly infer
    \begin{equation*}
        TGV^{\vphi,2}_\alpha(u) \le \alpha_2 \widetilde{V}_\vphi^n(Du) = \alpha_2 V_\vphi(u),
    \end{equation*}
    (see \eqref{def:EHH_TV}), 
    which, in turn, yields
    \begin{equation*}
        \|u\|_\vphi + TGV^{\vphi,2}_\alpha(u) \le \max\{1,\alpha_2 \} \|u\|_{BV^\vphi(\Omega)}.
    \end{equation*}
    Conversely, since $\widetilde{V}^n_\vphi$ satisfies the triangle inequality, letting $w_u$ be the minimum in \eqref{eq:dualTGV}, we deduce
    \begin{equation*}
        \alpha_2 V_{\vphi}(u) = \alpha_2 \widetilde{V}_\vphi^n(Du) \le \alpha_2 \widetilde{V}_\vphi^n(Du - w_u) + \alpha_2 \widetilde{V}_\vphi^n(w_u) \le TGV^{\vphi,2}_\alpha(u) + \alpha_2 \widetilde{V}_\vphi^n(w_u).
    \end{equation*}
\end{proof}

\begin{remark}\label{rem as thm 3.3}
    Thanks to Proposition \ref{prop:TGV_prop}, in particular thesis (i), and to Corollary \ref{cor:quasi_equiv}, we have that $TGV^{\vphi,2}_\alpha$ is a seminorm also on the space $BV^\vphi(\Omega)$.
\end{remark}

\begin{remark}
    Differently from the classical $TGV$-setting, it is not clear whether it is possible to show that there exists a constant $C > 0$ such that 
    \begin{equation*}
        \frac{1}{C} \|u\|_{BV^\vphi(\Omega)} \le \|u\|_\vphi + TGV^{\vphi,2}_\alpha(u)
    \end{equation*}
    or not. Roughly speaking, this is due to the absence of a Musielak-Orlicz equivalent of Sobolev-Korn inequality. We refer the interested reader to the discussion and the references in Section \ref{sec:funct_anal}.

   We also stress that, as it follows from the analysis in Proposition \ref{prop:PoincareTGV},
    $$\frac{1}{C} (\|u\|_{\vphi}+ |Du|(\Omega))\leq  \|u\|_{\vphi}+ TGV^{\vphi,2}_\alpha(u)
    $$
    for every $u \in BV^\vphi(\Omega)$.

\end{remark}

Thanks to what we have proven in Subsections \ref{sec:modular} and \ref{sec:BD_px}, we expect also the variation $\widetilde{V}^m_\vphi$ and, as a consequence, the total generalized variation $TGV^{\vphi,2}_\alpha$ to have a decomposition depending on the absolutely continuous part and the singular part of the gradient of the functions $u \in BV^\vphi(\Omega)$. We leave this open question for future works.

%\color{cyan} Elv: also this point is important to be answered soon in order to clarify the properties of the spaces.
\color{black}
%--------------------------------------------

\subsection{Existence and stability of solutions to the minimization problem}
\label{sec:ex_stab}

In this section we show how our notion of total generalized variation inherits classical existence and stability properties. Throughout this section we will always tacitly assume that $TGV^{\vphi,2}_\alpha$ is extended to $+\infty$ outside of its domain. We first prove that a Poincaré-Wirtinger inequality also holds in our setting, entailing the coercivity needed to prove the existence of solutions to minimum problems in image denoising.
\\
To this aim, from now on we denote by $\mathcal{P}^1(\Omega)$ the space of affine functions, as well as by $\Pi:L^1(\Omega)\to \mathcal{P}^1(\Omega)$ a linear projection.

\begin{proposition}
\label{prop:PoincareTGV}
    Let $\Omega \subset \R^n$ be a bounded, connected, open set with $C^1$ boundary. Let $\vphi \in \Phi_w(\Omega)$ satisfy (A0) and let $\alpha = (\alpha_1, \alpha_2) \in \R^+ \times \R^+$. Let $1 < p \le \tfrac{n}{n-1}$. Then, there exists a constant $c > 0$ such that
    \begin{equation*}
        \|u\|_{L^p(\Omega)} \le c \,TGV^{\vphi,2}_\alpha(u),
    \end{equation*}
    for every $u \in \ker\,\Pi \cap L^p(\Omega)$.
\end{proposition}
\begin{proof}
    Because $\vphi$ satisfies (A0), by Remark \ref{rmk:A0_si_no} we have $BV^\vphi(\Omega) \hookrightarrow BV(\Omega)$ and so
    \begin{equation*}
        |Du - w|(\Omega) \le \widetilde{V}_\vphi^n(Du - w).
    \end{equation*}
    Analogously, thanks to Remark \ref{rmk:A0_si_no_BD} we get $BD^\vphi(\Omega) \hookrightarrow BD(\Omega)$, entailing
    \begin{equation*}
        |Ew|(\Omega) \le \widetilde{V}_\vphi^{n\times n}(Ew),
    \end{equation*}
    for every $w \in BD(\Omega)$. In view of \cite[Theorem 3.1 \& Proposition 4.1]{BV20}, we infer the existence of a constant $c > 0$ such that
    \begin{equation*}
    \begin{aligned}
        \|u\|_{L^p(\Omega)} 
        &\le c \, TGV^{2}_\alpha(u) \le c \left[ \alpha_2 |Du - w|(\Omega) + \alpha_1 |Ew|(\Omega) \right] \\
        &\le c \left[ \alpha_2 \widetilde{V}_\vphi^n(Du - w) + \alpha_1 \widetilde{V}_\vphi^{n \times n}(Ew) \right],
    \end{aligned}
    \end{equation*}
    for every $w \in BD(\Omega)$. The thesis follows then by Theorem \ref{thm:equivalence_TGV}.
\end{proof}

%\color{cyan} Elv: Forse avrebbe senso scrivere la Proposizione precedente anche rispetto alla norma $L^\varphi$ che sembra pi\'u consistente, ma anche per questa servirebbe una Sobolev-Korn.
%\begin{proposition}
%\label{prop:PoincareTGVbis}
 %   Let $\Omega \subset \R^n$ be a bounded, connected, Lipschitz and open set. Let $\vphi \in \Phi_w(\Omega)$ satisfy (A0) and let $\alpha = (\alpha_1, \alpha_2) \in \R^+ \times \R^+.$ Then, there exists a constant $c > 0$ such that
  %  \begin{equation*}
   %     \|u\|_{L^\vphi(\Omega)} \le c \,TGV^{\vphi,2}_\alpha(u),
    %\end{equation*}
    %for every $u \in \ker\,\Pi \cap L^\vphi(\Omega)$.
%\end{proposition}
%\begin{proof}
%[Proof]
%If this is not true, there exists a sequence $(u_n) \in {\rm ker} P$
%with $\|u_n\|_\vphi = 1$ such that $1\geq C(n)TGV^{\alpha,2}_\vphi(u_n)$
%where $C(n) \geq n.$ We may assume $u_n \rightharpoonup u$ in %$L^\vphi(\Omega)$
 %with $u \in {\rm  ker} P$. According
%to Remark \ref{rem as thm 3.3}, $(u_n)$
%is also bounded in $BV(\Omega)$
%, thus
%we also have, by compact embedding, that $\lim u_n = u$
%in $L^1(\Omega)$. Lower semi-continuity now implies $TGV^{\alpha, 2}_\vphi(u)\leq \liminf_n TGV^{\alpha, 2}_\vphi(u_n) = 0$, 
%hence $u  \in \mathcal P^1(\Omega)\cap {\rm ker}P$ (see
%Theorem 2.2) and, consequently, $u = 0$. Thus, $u_n \overset{\ast}{\rightharpoonup} 0$ in
%$BV(\Omega)$ \color{magenta} qui ci vorrebbe $BV^\vphi$ ancora una volta usando Sobolev-Korn \color{cyan} and by continuous embedding, also in Lp(
%), which is
%a contradiction to kunkp = 1 for all n.
%\end{proof}

\color{black}
In view of the Lemma \ref{lmm:lsc_TGV}, we deduce existence of minimizers whenever $TGV^{\vphi,2}_\alpha$ is augmented by suitable fidelity terms. In what follows, given two Hilbert spaces $X$ and $H$, we denote by $\mathcal{L}(X;H)$ the set of linear and continuous operators between $X$ and $H$.

\begin{theorem}
    Let $\Omega \subset \R^n$ be a bounded, connected, open set with $C^1$ boundary. Let $\vphi \in \Phi_w(\Omega)$ satisfy (A0) and let $\alpha = (\alpha_1, \alpha_2) \in \R^+ \times \R^+$. Let $1 < p \le \tfrac{n}{n-1}$. Let further $H$ be an Hilbert space, and let $K \in \mathcal{L}(L^p(\Omega);H)$ be injective on $\mathcal{P}^1(\Omega)$. Then, for every $f \in H$, the minimum problem
    \begin{equation}
    \label{eq:minK}
        \min_{u \in L^p(\Omega)} \frac{1}{2} \|Ku-f\|_H + TGV^{\vphi,2}_\alpha(u)
    \end{equation}
    has a solution.
\end{theorem}
\begin{proof}
    Arguing exactly as in the proof of \cite[Theorem 4.2]{BV20}, in view of Propositions \ref{prop:TGV_prop} and \ref{prop:PoincareTGV}, we find that any minimizing sequence for \eqref{eq:minK} is uniformly bounded in $L^p(\Omega)$ for every $1\leq p \le \tfrac{n}{n-1}$. The thesis follows then by Lemma \ref{lmm:lsc_TGV}.
\end{proof}

We conclude this section by showing stability of minimizers of \eqref{eq:minK} with respect to the fidelity datum $f$.

\begin{theorem}
\label{thm:gamma-conv}
    Let $\Omega \subset \R^n$ be a bounded, connected, open set with $C^1$ boundary. Let $\vphi \in \Phi_w(\Omega)$ satisfy (A0) and let $\alpha = (\alpha_1, \alpha_2) \in \R^+ \times \R^+$. Let $1 < p \le \tfrac{n}{n-1}$. Let $H$ be an Hilbert space and let $K \in \mathcal{L}(L^p(\Omega);H)$ be injective on $\mathcal{P}^1(\Omega)$. Let $(f_k)_k \subset H$ be such that $f_k \to f$ strongly in $H$. Then, setting
    \begin{equation*}
        \mathcal{F}_k(u) := \frac{1}{2} \|Ku - f_k\|_H + TGV^{\vphi,2}_\alpha(u),
        \qquad
        \mathcal{F}_0(u) := \frac{1}{2} \|Ku - f\|_H + TGV^{\vphi,2}_\alpha(u),
    \end{equation*}
    for every $u\in L^p(\Omega)$, we have that $(\mathcal{F}_k)_k$ $\Gamma$-converges to $\mathcal{F}_0$ with respect to the weak $L^p$-topology.
\end{theorem}
\begin{proof}
    The limsup inequality is immediately verified by considering for each $u\in L^p(\Omega)$ the constant recovery sequence $\bar{u}_k = u$ for every $k \in \N$. To show the liminf inequality, let $u \in L^p(\Omega)$ and let $(u_k)_k \subset L^p(\Omega)$ be such that $u_k \rightharpoonup u$ weakly in $L^p(\Omega)$. The fact that
    \begin{equation*}
        \mathcal{F}_0(u) \le \liminf_{k \to +\infty} \mathcal{F}_k(u_k)
    \end{equation*}
    is a direct consequence of Lemma \ref{lmm:lsc_TGV}.
\end{proof}

As a consequence of Theorem \ref{thm:gamma-conv} we infer convergence of the associated minimizers.

\begin{corollary}
    Under the same hypotheses of Theorem \ref{thm:gamma-conv}, we have
    \begin{itemize}
        \item[(i)] $\min \mathcal{F}_k \to \min \mathcal{F}_0$.
        
        \item[(ii)] Letting $u_k \in \textnormal{argmin }\mathcal{F}_k$ for every $k \in \N$, we have that $(u_k)_k$ is precompact in $L^{\frac{n}{n-1}}(\Omega)$, and that every limiting point $u$ satisfies $u \in \textnormal{argmin }\mathcal{F}_0$.
    \end{itemize}
\end{corollary}
\begin{proof}
    The result follows from the Fundamental Theorem of Gamma-convergence, cf. \cite{B}, by combining Theorem \ref{thm:gamma-conv} and Proposition \ref{prop:PoincareTGV}.
\end{proof}

%%%%%%%%%%%%%%%%%%%%%%%%%%%%%%%%%%%%%%%%%%%%%%%%%%%%

\appendix

\section{Proofs of the results in Subsection \ref{sec:modular}}
\label{app:proof_mod}

In Proposition \ref{prop:singularPart} and Proposition \ref{prop:exactFormulaAC} we consider respectively the singular and absolutely continuous parts of the symmetric gradient separately and we then combine them to handle the whole function in Theorem \ref{thm:exactFormula}.
\\
In order to prove Proposition \ref{prop:singularPart}, we need the following result, analogous to \cite[Lemma 3.5]{EHH}, which holds a characterization of the singular part of the symmetric gradient of our function.

\begin{lemma}
\label{lem:singularPart}
    Let $\Omega \subset \R^n$ be a bounded, connected and open set. If $u \in BD(\Omega)$, then
    \begin{equation*}
        |E^s u|(\Omega) = \sup \left\{ \int_\Omega \psi : dE^su \st \psi \in C^1_c(\Omega; \R^{n \times n}_\sym), \, |\psi| \le 1 \right\}.
    \end{equation*}
\end{lemma}
\begin{proof}
    By the definitions of total variation of a measure and of weak derivative, together with the fact that $Eu = \E u + E^s u$, we see that
    \begin{equation*}
    \begin{aligned}
        |Eu|(\Omega)
        &= \sup_{\psi \in C^1_c(\Omega; \R^{n \times n}_\sym), |\psi| \le 1} \left( \int_\Omega \psi : d\E u + \int_\Omega \psi : dE^su \right) \\
        &\le \sup_{\psi \in L^\infty(\Omega; \R^{n \times n}_\sym), |\psi| \le 1} \int_\Omega \psi : d \E u + \sup_{\psi \in L^\infty(\Omega; \R^{n \times n}_\sym), |\psi| \le 1} \int_\Omega \psi : dE^su \\
        &= 
        |\E u|(\Omega) + |E^su|(\Omega). 
    \end{aligned}
    \end{equation*}
    Since $|\E u|(\Omega) + |E^su|(\Omega) = |Eu|(\Omega)$ as $\E$ and $E^s$ are mutually singular, each inequality has to be an equality, and so the claim follows.
\end{proof}

We can now give the proof of Proposition \ref{prop:singularPart}, following the lines of \cite[Proposition 6.2]{EHH}.

\begin{proof}[Proof of Proposition \ref{prop:singularPart}]
    By the definition of the total variation of a measure, Lemma \ref{lem:bound} and the definition of $T^\vphi$, see \eqref{def_test}, we have
    \begin{equation*}
        \sup_{\psi \in T^\vphi} \int_\Omega \psi : dE^su 
        \le \sup_{\psi \in T^\vphi} \int_\Omega |\psi| \, d|E^su| 
        \le \int_\Omega \varphi'_\infty \, d|E^su|.
    \end{equation*}
    In order to tackle the opposite inequality, let us define $h_k : \Omega \to [0,+\infty]$ by
    \begin{equation*}
        h_k(x) := \lim_{r \to 0^+} \inf_{y \in B(x,r)} \frac{\vphi(y, k)}{k} \qquad \textnormal{ for any } x \in \Omega.
    \end{equation*}
    Notice that $h_k$ is lower semicontinuous with 
    \begin{equation}\label{leqhk}
    h_k \le \tfrac{\vphi(\cdot, k)}{k} \le \vphi^\infty.
    \end{equation}
    From the first inequality it follows that $\vphi^*(\cdot, h_k) \le \vphi(\cdot, k)$, implying that $\vrho_{\vphi^*}(h_k) \le \vrho_\vphi(k) < +\infty$ since $\vphi$ satisfies (A0) and (aDec), and $\Omega$ is bounded.
    \\
    Let us now remark that $h_k \to \vphi^\infty$ as $k \to +\infty$. Indeed, if $\vphi^\infty(x) = +\infty$, then we can use (A1) in all sufficiently small balls, since $\vphi^+(k) < +\infty$, to conclude that 
    \begin{equation}\label{geqhk}
        h_k(x) = \lim_{r \to 0^+} \inf_{y \in B(x,r)} \frac{\vphi(y, k)}{k} \ge \frac{\vphi(x, \beta k) - 1}{k} \to \beta \vphi^\infty(x) = +\infty
    \end{equation}
    as $k \to +\infty$. If instead $\vphi^\infty(x)< +\infty$, then we use the same inequality but now with $\beta := \tfrac{1}{1 + \omega(r)}$, where $\omega$ is the modulus of continuity coming from the restricted (VA1) condition, and we obtain the desired convergence as $\beta \to 1^-$. 
    \\
    Note that $h_k$ is increasing in $k$ since $\vphi$ is convex. Then, by monotone convergence,
    \begin{equation*}
        \int_\Omega \vphi^\infty \, d|E^su| = \lim_{k \to \infty} \int_\Omega h_k \, d|E^su|.
    \end{equation*}
    Let us fix $\e > 0$ and assume that $\int_\Omega \vphi^\infty \, d|E^su| < +\infty$. Then, we can find $h = h_k$ and $K > 0$ such that 
    \begin{equation*}
    \begin{aligned}
        \int_\Omega \vphi^\infty \, d|E^su|
        &\le \int_\Omega h \, d|E^su| + \e
        \le \sum_{j = 1}^{K^2} \int_\Omega \frac{1}{K} \, \chi_{\{h > \tfrac{j}{K}\}} \, d|E^su| + 2\e \\
        &= \frac{1}{K} \sum_{j = 1}^{K^2} |E^su|\left(\left\{h > \frac{j}{K}\right\}\right) + 2\e.
    \end{aligned}
    \end{equation*}
    Since $h$ is lower semicontinuous, then $\{h > \tfrac{j}{K}\}$ is open, and hence by Lemma \ref{lem:singularPart} we can choose $\psi_j \in  C^1_c(\Omega; \R^{n \times n}_\sym)$ with $|\psi_j| \le 1$ such that 
    \begin{equation*}
        \int_\Omega \vphi^\infty \, d|E^su|
        \le \frac{1}{K} \sum_{j = 1}^{K^2} \int_{\{h > \tfrac{j}{K}\}} \psi_j : dE^su + 3\e
        = \int_\Omega \left(\sum_{j = 1}^{K^2} \frac{1}{K} \psi_j \right) : dE^su + 3\e.
    \end{equation*}
    Defining $\psi_\e := \sum_{j = 1}^{K^2} \tfrac{1}{K} \psi_j$, we can notice that $\psi_\e \in  C^1_c(\Omega; \R^{n \times n}_\sym)$ and 
    \begin{equation*}
        |\psi_\e| \le \frac{1}{K} \sum_{j = 1}^{K^2} |\psi_j| \le \frac{1}{K} \sum_{j = 1}^{K^2} \chi_{\{h > \frac{j}{K}\}} \le h,
    \end{equation*}
    giving $\vrho_{\vphi^*}(|\psi_\e|) < +\infty$. Therefore $\psi_\e \in T^\vphi$ and 
    \begin{equation*}
        \int_\Omega \vphi^\infty \, d|E^su|
        \le \int_\Omega \psi_\e : dE^su + 3\e
        \le \sup_{\psi \in T^\vphi} \int_\Omega \psi : dE^su +3\e.
    \end{equation*}
    The upper bound then follows letting $\e \to 0^+$. If instead $\int_\Omega \vphi^\infty \, d|E^su| = +\infty$, then a similar argument gives $\tfrac{1}{3\e} \le \sup_{\psi \in T^\vphi} \int_\Omega \psi : dE^su$ and the claim again follows.
\end{proof}

Lastly, we now give the proof of Proposition \ref{prop:exactFormulaAC}, following the lines of \cite[Proposition 6.3]{EHH}.

\begin{proof}[Proof of Proposition \ref{prop:exactFormulaAC}]
    The upper bound follows directly from Young's inequality, indeed
    \begin{equation*}
        \sup_{\psi \in T^\vphi} \int_\Omega [\E u : \psi - \vphi^*(x, |\psi|)] \, dx \le \int_\Omega \vphi(x, |\E u|) \, dx = \vrho_\vphi(|\E u|).
    \end{equation*}
    In order to prove the lower bound let us choose $g_i \in C(\Omega; \R^{n \times n}_\sym) \cap L^\vphi(\Omega; \R^{n \times n}_\sym)$ such that $g_i \to \E u$ pointwise a.e and in $L^1(\Omega; \R^{n \times n}_\sym)$. Then Fatou's Lemma and $L^1$-convergence yield
    \begin{equation*}
        \int_\Omega \vphi(x, |\E u|) \, dx \le \liminf_{i \to +\infty} \int_\Omega \vphi(x, |g_i|) \, dx
        \quad \textnormal{ and } \quad
        \lim_{i \to +\infty} \int_\Omega g_i : \psi \, dx = \int_\Omega \E u : \psi \, dx 
    \end{equation*}
    for any $\psi \in T^\vphi$. Thus, it suffices to show that 
    \begin{equation*}
        \int_\Omega \vphi(x, |g|) \, dx \le \sup_{\psi \in T^\vphi} \int_\Omega [g : \psi - \vphi^*(x, |\psi|)] \, dx
    \end{equation*}
    for any $g \in C(\Omega; \R^{n \times n}_\sym) \cap L^\vphi(\Omega; \R^{n \times n}_\sym)$. Moreover, replacing $\psi$ by $\frac{g}{\e + |g|} |\psi|$ and letting $\e \to 0^+$, we see that $g :  \frac{g}{\e + |g|} |\psi| \to |g| \, |w|$ pointwise. Thus, by monotone convergence, the vector-values of $g$ and $\psi$ are irrelevant and we only need to show that 
    \begin{equation}
    \label{eq:thesis_2}
        \int_\Omega \vphi(x, |g|) \, dx \le \sup_{\psi \in C^1_c(\Omega)} \int_\Omega [|g \psi| - \vphi^*(x, |\psi|)] \, dx 
    \end{equation}
    for any $g\in C(\Omega) \cap L^\vphi(\Omega)$. We can also exclude test functions such that $\vrho_{\vphi^*}(\psi) = +\infty$, see Remark \ref{rmk:test_mod}.
    \\
    Now let us denote with $\vphi'$ the left-derivative of $\vphi$ with respect to the second variable. Then $\vphi'$ is left-continuous and $\vphi(x,s) \ge \vphi(x,s_0) + \vphi'(x,s_0)(s-s_0)$ by convexity. For any $\e > 0$, let us define 
    \begin{equation}
    \label{final_choice}
        \psi_\e(x,t) := \int_{-\infty}^{+\infty} \vphi(x, \max\{\tau,0\}) \, \zeta_\e(t-\tau) \, d\tau = (\vphi *_t \zeta_\e)(x, t),
    \end{equation}
    where $\zeta_\e$ is a mollifier in $\R$ with support in $[0,\e]$ and where $*_t$ denotes the convolution only in the second variable. Since $\vphi$ and $\vphi'$ are increasing in the second variable and left-continuous, we have that $\psi_\e \nearrow \vphi$ and $\psi_\e' \nearrow \vphi'$ as $\e \to 0^+$. Moreover, $\psi_\e' = \vphi *_t \zeta_\e'$ is continuous in both $x$, since $\vphi$ is, and $t$, as a convolution with a smooth function.
    \\
    Let $v_i \in C_c(\Omega)$ with $0 \le v_i \le 1$ and $v_i \nearrow 1$ as $i \to +\infty$. By uniform continuity in $ \supp \, v_i$, we can choose $\delta = \delta_{\e, i} > 0$ such that, for any $\e > 0$ and $i \in \N$,
    \begin{equation*}
        \psi_\e'(x, |g(x)| \, v_i(x)) - \e \le \psi_\e'(y, |g(y)| \, v_i(y))
    \end{equation*}
    for all $x \in B(y,\delta)$ and $y \in \Omega$. Then
    \begin{equation*}
        \psi_{\e, i} := \max\{\psi_\e'(\,\cdot\,, |g| \, v_i) - \e, 0\} *_x \eta_\delta \le \psi_\e'(\cdot, |g|) \le \vphi'(\cdot, |g|),
    \end{equation*}
    where $\eta_\delta$ is a mollifier in $\R$ with support in $[0,\delta]$ and where $*_x$ denotes the convolution only in the first variable. We have that $\psi_{\e,i} \to \vphi'(\cdot, |g|)$, so we conclude by Fatou's Lemma that
    \begin{equation*}
        \int_\Omega |g| \, \vphi'(x, |g|) \, dx \le \liminf_{i \to +\infty, \, \e \to 0} \int_\Omega |g \, \psi_{\e, i}| \, dx.
    \end{equation*}
    Since $\vphi$ satisfies (A0) and (aDec), we see that
    \begin{equation*}
        \vphi^*(x, |\psi_{\e,i}|) \le \vphi^*(x, \vphi'(x, |g|)) \le |g| \, \vphi'(x, |g|) \le c \, \vphi(x, |g|),
    \end{equation*}
    where $c > 0$ is a constant. Since $g \in L^\vphi(\Omega)$ and $\vphi$ satisfies (aDec), then $\vrho_\vphi(g) < +\infty$. Thus, dominated convergence with upper bound $c \, \vphi(\cdot, g)$ yields
    \begin{equation*}
        \int_\Omega \vphi^*(x, \vphi'(x, |g|)) \, dx = \lim_{i \to +\infty, \, \e \to 0} \int_\Omega \vphi^*(x, |\psi_{\e,i}|) \, dx.
    \end{equation*}
    Since $\psi_\e$ is a valid test function and $\vphi'(\cdot, |g|) < +\infty$ a.e. in $\Omega$, the last equality, together with ``Young's equality'', see \cite[Lemma 1.7.3(i)]{NicP18}, implies that
    \begin{equation*}
    \begin{aligned}
        \sup_{\psi \in C^1_c(\Omega), \, \vrho_{\vphi^*}(\psi) < +\infty} &\int_\Omega [|g \psi| - \vphi^*(x, |\psi|)] \, dx \ge \liminf_{i \to +\infty, \, \e \to 0} \int_\Omega [|g| \, |\psi_{\e,i}| - \vphi^*(x, |\psi_{\e,i}|)]\, dx \\
        \ge &\int_\Omega [|g| \, \vphi'(x, |g|) - \vphi^*(x, \vphi'(x, |g|))] \, dx = \int_\Omega \vphi(x, |g|) \, dx,
    \end{aligned}
    \end{equation*}
    proving \eqref{eq:thesis_2} and completing the proof of the lower bound.
\end{proof}

\begin{remark}
    The natural idea would be to choose $\psi := \vphi'(x, |g|)$ in the supremum in \eqref{eq:thesis_2}, instead of the family of test functions that we chose in \eqref{final_choice}, where $\vphi'$ is the left-derivative of $\vphi$ with respect to the second variable. However, this function is not in general smooth and we cannot use regular approximation by smooth functions since $\vphi^*$ is not doubling.
\end{remark}

%%%%%%%%%%%%%%%%%%%%%%%%%%%%%%%%%%%%%%%%%%%%%%%%%%%%

\section{A Convex Analysis lemma}
\label{AmarLemma}

The following result is presented for readers convenience and it is due to \cite{Aapp}, see Lemma 2.33 therein. It has been used in Section \ref{sec:funct_anal}.

\begin{lemma}
\label{lemmaApp}
    Let $X$ be a Banach space with norm $\|\cdot \|_X$, let $\phi : \R \to \R$ be an even and convex function, and define $f : X \to \R$ as $f(x) := \phi(\|x\|_X)$. Then
    \begin{equation*}
        f^*(x^*) = \phi^*(\|x^*\|_{X^*})
    \end{equation*}
    where $X^*$ denotes the dual of $X$ with norm $\|\cdot\|_{X^*}$.
\end{lemma}
\begin{proof}
    By definition
    \begin{equation*}
        f^*(x^*) = \sup_{x \in X} [\langle x^*, x \rangle - f(x)] = \sup_{x \in X} [\langle x^*, x \rangle - \phi(\|x\|_{X})],
    \end{equation*}
    which can be equivalently written as
    \begin{equation}
    \begin{aligned}
    \label{2.13Amar}
        f^*(x^*)
        &= \sup_{t \ge 0} \sup_{\|x\|_X = t} [\langle x^*, x \rangle - \phi(t)] \\
        &= \sup_{t \ge 0} \left\{ \left[ \sup_{\|x\|_X = t} \langle x^*, x \rangle \right] - \phi(t) \right\} = \sup_{t \ge 0} [t \|x^*\|_{X^*} - \phi(t)].
    \end{aligned}
    \end{equation}
    On the other hand, by definition
    \begin{equation}
    \label{2.14Amar}
        \phi^*(\|x^*\|_{X^*}) = \sup_{t \in \R} [t \|x^*\|_{X^*} - \phi(t)],
    \end{equation}
    hence it suffices to prove that \eqref{2.13Amar} coincides with \eqref{2.14Amar}, i.e. that it suffices to take the supremum on nonnegative real numbers. Using the fact that $\phi$ is even, for any $t > 0$ we have
    \begin{equation*}
        -t \|x^*\|_{X^*} - \phi(-t) = -t \|x^*\|_{X^*} - \phi(t) \le t \|x^*\|_{X^*} - \phi(t),
    \end{equation*}
    hence
    \begin{equation*}
        \sup_{t \in \R} [t\|x^*\|_{X^*} - \phi(t)] = \sup_{t \ge 0} [t\|x^*\|_{X^*} - \phi(t)],
    \end{equation*}
    which concludes the proof.
\end{proof}

%%%%%%%%%%%%%%%%%%%%%%%%%%%%%%%%%%%%%%%%%%%%%%%%%%%

\section*{Acknowledgements}

The research of E. Davoli and S. Ricc\'o was funded by the Austrian Science Fund (FWF) projects \href{https://www.doi.org/10.55776/F65}{10.55776/F65}, \href{https://www.doi.org/10.55776/Y1292}{10.55776/Y1292}, \href{https://www.doi.org/10.55776/P35359}{10.55776/P35359}, and \href{https://www.doi.org/10.55776/F100800}{10.55776/F100800}. For open-access purposes, the authors have applied a CC BY public copyright license to any author-accepted manuscript version arising from this submission. G. Bertazzoni and E. Zappale are members of INdAM GNAMPA, whose support is gratefully acknowledged. They are also very grateful to the Institute of Analysis and Scientific Computing at the TU Wien for its kind support and hospitality during the preparation of this work. E. Zappale's research has also been supported by PRIN 2022: Mathematical Modelling of Heterogeneous Systems (MMHS) - Next Generation EU CUP B53D23009360006.

%%%%%%%%%%%%%%%%%%%%%%%%%%%%%%%%%%%%%%%%%%%%%%%%%%%

\end{document}